\documentclass[journal,twoside,web]{ieeecolor}
\usepackage{generic}
\usepackage{textcomp}
\def\BibTeX{{\rm B\kern-.05em{\sc i\kern-.025em b}\kern-.08em
    T\kern-.1667em\lower.7ex\hbox{E}\kern-.125emX}}
\usepackage[T1]{fontenc}
\usepackage[english]{babel}
\usepackage[latin1]{inputenc}
\usepackage{cite}
\usepackage{amsmath,amsfonts,amssymb}
\usepackage{mathtools}
\usepackage{graphicx}
\usepackage{url}
\usepackage{hyperref}
\usepackage{cleveref}
\usepackage[margin=1in]{geometry}
\usepackage{authblk}
\usepackage{float} 
\usepackage{soul} 

\usepackage{algorithm,algpseudocode}

\usepackage{etoolbox}
\makeatletter
\@ifundefined{color@begingroup}%
  {\let\color@begingroup\relax
   \let\color@endgroup\relax}{}%
\def\fix@ieeecolor@hbox#1{%
  \hbox{\color@begingroup#1\color@endgroup}}
\patchcmd\@makecaption{\hbox}{\fix@ieeecolor@hbox}{}{\FAILED}
\patchcmd\@makecaption{\hbox}{\fix@ieeecolor@hbox}{}{\FAILED}

\newcommand{\bit}{\begin{itemize}}
\newcommand{\eit}{\end{itemize}}
\newtheorem{theorem}{Theorem}
\newtheorem{proposition}{Proposition}

\newtheorem{definition}{Definition}

\newtheorem{remark}{Remark}

\newcommand{\Real}{\mathbb{R}}

\newcommand{\bSigma}{\mathbf{\Sigma}}
\newcommand{\bXi}{\mathbf{\Xi}}
\newcommand{\bPsi}{\mathbf{\Psi}}

\DeclareMathOperator{\diag}{diag}

\newcommand{\bzero}{\mathbf{0}}
\renewcommand{\a}{\mathbf{a}}
\renewcommand{\b}{\mathbf{b}}
\renewcommand{\c}{\mathbf{c}}

\newcommand{\f}{\mathbf{f}}
\newcommand{\g}{\mathbf{g}}
\newcommand{\h}{\mathbf{h}}

\renewcommand{\u}{\mathbf{u}}
\renewcommand{\v}{\mathbf{v}}
\newcommand{\w}{\mathbf{w}}
\newcommand{\x}{\mathbf{x}}
\newcommand{\y}{\mathbf{y}}
\newcommand{\z}{\mathbf{z}}

\newcommand{\A}{\mathbf{A}}
\newcommand{\B}{\mathbf{B}}
\newcommand{\C}{\mathbf{C}}

\newcommand{\I}{\mathbf{I}}
\newcommand{\J}{\mathbf{J}}
\renewcommand{\L}{\mathbf{L}}
\newcommand{\M}{\mathbf{M}}
\newcommand{\N}{\mathbf{N}}

\newcommand{\R}{\mathbf{R}}

\newcommand{\T}{\mathbf{T}}

\newcommand{\V}{\mathbf{V}}
\newcommand{\W}{\mathbf{W}}
\newcommand{\X}{\mathbf{X}}
\newcommand{\Y}{\mathbf{Y}}

\newcommand{\cE}{\mathcal{E}}
\newcommand{\cH}{\mathcal{H}}

\newcommand{\cT}{\mathcal{T}}
\newcommand{\cW}{\mathcal{W}}

\usepackage[textsize=tiny]{todonotes}

\newcommand*\circled[1]{\tikz[baseline=(char.base)]{
\node[shape=circle, draw, inner sep=0.6pt] (char) {#1};}
}
\newcommand\kronF[2]{#1^{\circled{\tiny{#2}}}}

\begin{document}
\title{Nonlinear Balanced Truncation:\\ 
Part 2---Model Reduction on Manifolds}
\author{Boris Kramer, Serkan Gugercin, and Jeff Borggaard
\thanks{This work was supported in part by the NSF under Grant CMMI-2130727 and is based upon work supported by the National Science Foundation under Grant No. DMS-1929284 while the authors were in residence at the Institute for Computational and Experimental Research in Mathematics in Providence, RI, during the Spring 2020 Semester Program "Model and dimension reduction in uncertain and dynamic systems" and Spring 2020 Reunion Event.}
\thanks{B. Kramer is with the Department of Mechanical and Aerospace Engineering, University of California San Diego, La Jolla, CA 92093-0411 USA (e-mail: bmkramer@ucsd.edu).}
\thanks{S. Gugercin and J. Borggaard are with the Department of Mathematics, Virginia Tech, Blacksburg, VA 24061 (e-mail: gugercin@vt.edu, jborggaard@vt.edu).}
}

\maketitle

\begin{abstract}
Nonlinear balanced truncation is a model order reduction technique that reduces the dimension of nonlinear systems in a manner that accounts for either open- or closed-loop observability and controllability aspects of the system. Two computational challenges have so far prevented its deployment on large-scale systems: (a) the energy functions required for characterization of controllability and observability are solutions of high-dimensional Hamilton-Jacobi-(Bellman) equations, which have been computationally intractable and (b) the transformations to construct the reduced-order models (ROMs) are potentially ill-conditioned and the resulting ROMs are difficult to simulate on the nonlinear balanced manifolds.
Part~1 of this two-part article (\!\!~\cite{KGB_NonlinearBT_Part1}) addressed challenge (a) via a scalable tensor-based method to solve for polynomial approximations of the open- and closed-loop energy functions.
This article, (Part~2), addresses challenge (b) by presenting a novel and scalable method to reduce the dimensionality of the full-order model via model reduction on polynomially-nonlinear balanced manifolds. The associated nonlinear state transformation simultaneously ``diagonalizes'' relevant energy functions in the new coordinates. Since this nonlinear balancing transformation can be ill-conditioned and expensive to evaluate, inspired by the linear case we develop a computationally efficient \textit{balance-and-reduce} strategy, resulting in a scalable and better conditioned truncated transformation to produce balanced 
ROMs. 
The algorithm is demonstrated on a semi-discretized partial differential equation, namely Burgers equation, which illustrates that higher-degree transformations can improve the accuracy of ROM outputs.
\end{abstract}

\begin{IEEEkeywords}
Reduced-order modeling, balanced truncation, nonlinear manifolds, Hamilton-Jacobi-Bellman equation, nonlinear control-affine systems.
\end{IEEEkeywords}

\section{Introduction} \label{sec:intro}
Simulation of large-scale nonlinear dynamical systems can be time-consuming and resource-intensive.  It is a frequent bottleneck when these simulations are used for real-time, model-based control. Reduced-order models (ROMs) provide an attractive solution to this problem by approximating dynamical systems (and their relevant system-theoretic properties) in a much lower dimensional state space; see, e.g., \cite{AntBG20, antoulas05, BCOW2017morBook, Zho96, SvdVR08,BGW15surveyMOR} for a general overview. These ROMs then allow for the design of low-dimensional controllers and filters. 

Balanced truncation model reduction, as pioneered by Moore~\cite{moore81principal} and
Mullis and Roberts~\cite{mullis1976synthesis}
for \textit{linear} time-invariant (LTI) systems, provides an elegant approach to the model reduction problem for open-loop settings. The approach uses the controllability and observability energies of a system to determine those states that have the most relative importance. If a state requires a large amount of input energy to be reached and also has minimal effect on the output, then the reduced model would neglect that state without a significant impact on the input-output behavior of the system. Extensions of this concept to closed-loop LTI systems led to the LQG balancing \cite{verriest1981suboptimal,jonckheere1983LQGbalancing} and $\mathcal{H}_\infty$ balancing concepts~\cite{glover91Hinftybalancing}. 
Several variants for different formulations of the LTI system and different reduced-model outcomes exist, such as stochastic balancing~\cite{desai1984stochasticBal,green1988balancedstochastic}, bounded real balancing~\cite{opdenacker1988contraction}, positive real balancing~\cite{desai1984stochasticBal,ober1991balanced}, and frequency-weighted balancing~\cite{enns1984model}; see also the surveys~\cite{gugercin2004survey,benner2017chapter}. The interest in balancing methods for LTI systems in the 1980s stimulated research in computational methods for solving large-scale Lyapunov or Riccati-type algebraic matrix equations that yielded low-rank solvers \cite{benner2013numerical,simoncini2016computational} and {doubling methods~\cite{li2013solving}} that can solve these matrix equations for millions of states.

For \textit{nonlinear} large-scale systems, the theory is largely developed, yet computationally scalable approaches and efficient ROM development remain open problems. 
The theoretical foundation for balanced truncation of nonlinear open-loop systems proposed by Scherpen~\cite{scherpen1993balancing} defines input and output energy functions and shows that they can be computed as solutions to Hamilton-Jacobi ({HJ}) partial differential equations (PDEs). Theoretical extensions of~\cite{scherpen1993balancing} to the closed-loop setting, such as Hamilton-Jacobi-Bellman (HJB) balancing~\cite{scherpen1994normCoprime}  and $\mathcal{H}_\infty$ balancing~\cite{scherpen1996hinfty_balancing} have also been proposed for nonlinear systems, which can be unstable.  
Aside from the need for computing solutions of HJ(B) equations, nonlinear balancing requires a {\em nonlinear transformation} whereas balancing for LTI systems requires a linear change of variables. While reduced models obtained through the original nonlinear balancing framework~\cite{scherpen1993balancing} were quickly shown to be non-unique~\cite{gray2001nonuniqueness}, a slightly altered transformation suggested in~\cite{fujimoto2010balanced} resolved this issue, resulting in uniquely determined reduced-order models. Krener~\cite{krener2008reduced} suggested a  different balancing strategy that applies the balancing to each degree of the polynomial approximations of the energy functions separately.
The full nonlinear balance-then-reduce approaches first compute the balancing transformation in the high-dimensional space, and subsequently reduce the model dimension, see~\cite{scherpen1993balancing,scherpen1994normCoprime,scherpen1996hinfty_balancing,fujimoto2008TaylorSeriesBT,sahyoun2013reduced}. However, in the linear case, balance-then-reduce approaches have been shown to be ill-conditioned due the need to invert small Hankel singular values~\cite{antoulas05}.
In sum, the nonlinear transformation is difficult to compute---reducing the efficiency of the nonlinear ROM---and can be ill-conditioned as well.
Lastly, the development of an error bound similar to the linear case (in terms of neglected Hankel singular values) remains an open problem.

Other methods for reducing the dimensionality of the FOM via balancing have been proposed that avoid the computation of the fully nonlinear energy function. Generally, these methods compute a quadratic energy function that then leads to linear subspace reduction. For instance, \cite{verriest2006algebraicGramiansNLBT,bennerGoyal2017BT_quadBilinear} derive algebraic Gramians for the nonlinear system and propose a linear coordinate transform. The authors in \cite{condon2005nonlinear} combined Carleman bilinearization with a balancing method to approximately balance weakly nonlinear systems. Empirical Gramians for nonlinear systems can also be used~\cite{lall2002subspace} and also lead to linear subspace reduction. Lastly, a method for approximating the nonlinear balanced truncation reduction map via reproducing kernel Hilbert spaces (a machine learning-based, data-driven technique) has been proposed in~\cite{bouvrie2017kernel}, and applied to two- and seven-dimensional ODE examples. The idea is to embed the nonlinear dynamical system into a high (or infinite) dimensional Reproducing Kernel Hilbert Space (RKHS). There, linear theory can be applied. Then, the authors learn mappings from the high dimensional RKHS to the balanced finite dimensional system and demonstrate their results on 2d and 7d ODE examples. 
There  also exist related methods that balance different characteristics of the problem (not the aforementioned nonlinear energy functions), such as flow balancing~\cite{verriest2000flow}, incremental balancing~\cite{besselinkScherpen2014incrementalBT}, dynamic balancing~\cite{sassano2014dynamic}. These methods are only loosely connected to the nonlinear Hankel operator and the energy functions~\cite{kawano2016model}. Moreover, incremental balancing is restricted to only odd functions on the right-hand-side of the dynamical systems. Another type of balancing, called differential balancing~\cite{kawano2016model}, has been developed with a connection to the nonlinear Hankel operator. This method requires solving linear time-dependent matrix inequalities (extensions of Lyapunov inequalities) and is therefore focused on open-loop systems.

In this two-part article, we propose a scalable and computationally efficient nonlinear balanced truncation approach via nonlinear energy functions that works for both closed- and open-loop systems. Such a method is currently lacking for medium- and large-scale systems. 
In Part~1 of this article, \cite{KGB_NonlinearBT_Part1}, we (i) propose a unifying framework to the open- \textit{and} closed-loop nonlinear balancing problem by considering Taylor-series-based approaches to solve a class of parametrized HJB equations; (ii) derive the explicit tensor structure for the coefficients of the polynomial expansion which allows our numerical methods to scale up to thousands of state variables; (iii) provide large-scale numerical examples, a scalability analysis, and open-access software for all algorithms.

There are three main contributions of this article (Part~2). 
First, we present a scalable Taylor-series-based strategy that is implemented in open-access software (the \texttt{NLbalancing} repository~\cite{NLbalancing}) for computing the nonlinear state transformations and singular value functions. We then derive the nonlinearly balanced models. The nonlinear state transformation simultaneously ``diagonalizes'' the open- and closed-loop energy functions in the new coordinates. However, this nonlinear balancing transformation can be ill-conditioned (which also occurs in the LTI case) and is expensive to evaluate. 
Second, we thus develop a computationally efficient \textit{balance-\underline{and}-reduce} strategy that reduces the dimensional of the full-order model (FOM) via model reduction on nonlinear manifolds. The strategy addresses the ill-conditioning problem in the full nonlinear transformation and allows for efficient nonlinear ROM construction. 
Third, we present nonlinear balanced ROMs for a semi-discretized PDE of Burgers-type, where we investigate the impact of the degree of the nonlinear transformation on the ROM input-output quality and illustrate the behavior of the singular value functions.
We highlight that the work herein does not assume a specific form of the dynamical system: the transformation merely balances two polynomial energy functions and produces a nonlinear mapping that can be applied to any nonlinear system.

This article is organized as follows. Section~\ref{sec:NLBal} briefly reviews background material on energy functions for nonlinear systems. Section~\ref{sec:BalModels} presents our first key contribution, a scalable strategy for computing the nonlinear basis transformations and the singular value functions. We also present the nonlinearly balanced models in that section. Section~\ref{sec:ROMs} presents our second key contribution, a modified and scalable method to obtain a balanced nonlinear ROM that addresses the ill-conditioning problem in the full nonlinear transformation and allows for more efficient nonlinear ROM evaluation.  Section~\ref{sec:numerics} presents numerical results for the nonlinear ROMs for the semi-discretized Burgers equation.
Lastly, Section~\ref{sec:conclusions} offers conclusions and an outlook toward future work.

\section{Energy functions for nonlinear balancing} \label{sec:NLBal}
This section briefly reviews the unifying concept of the $\cH_\infty$ energy functions for nonlinear control-affine dynamical systems, see Part~1 of this article (\hspace{-.01em}\cite{KGB_NonlinearBT_Part1}) for more details. 
Consider the finite-dimensional system
\begin{align}
\dot{\x}(t) & = \f(\x(t))  + \g(\x(t)) \u(t), \label{eq:FOMNL1}\\
\y(t) &= \h(\x(t)),\label{eq:FOMNL2}
\end{align}
where $t$ denotes time, $\x(t) \in \Real^{n}$ is the state, 
$\u(t) \in \Real^m$ is a time-dependent input vector, 
$\g:\Real^n\mapsto\Real^{n\times m}$  encodes the actuation mechanism,
$\y(t) \in \Real^p$ is the output vector measured by the function $\h:\Real^n \mapsto \Real^p$, the nonlinear drift term is $\f:\Real^n\mapsto\Real^n$, and $\x =\bzero$ is an isolated equilibrium point for $\u=\bzero$.
For the  system
\eqref{eq:FOMNL1}-\eqref{eq:FOMNL2}, the $\mathcal{H}_\infty$ past energy in the state $\x_0$ is defined, for $0<\gamma$, $\gamma\neq 1$, as
{\small \begin{equation} \label{eq:HinftyenergyFunctions1}
\cE_\gamma^{-}(\x_0)  :=\min_{\substack{\u \in L_{2}(-\infty, 0] \\ \x(-\infty) = \bzero, \\  \x(0) = \x_0}} \ \frac{1}{2} \int\displaylimits_{-\infty}^{0} (1-\gamma^{-2})\Vert \y(t) \Vert^2  +  \Vert \u(t) \Vert^2 {\rm{d}}t.
\end{equation}}
Furthermore, the  $\mathcal{H}_\infty$ future energy in the state $\x_0$ is defined, 
for $0<\gamma<1$ as 
\begin{equation} \label{eq:HinftyenergyFunctions3}
\cE_\gamma^{+}(\x_0)  :=\max_{\substack{\u \in L_{2}[0,\infty) \\ \x(0) = \x_0, \\  \x(\infty) = \bzero}} \ \frac{1}{2} \int\displaylimits_{0}^{\infty} \Vert \y(t) \Vert^2  +  \frac{\Vert \u(t) \Vert^2}{1-\gamma^{-2}} {\rm{d}}t,
\end{equation}
and for $\gamma>1$, as
\begin{equation} \label{eq:HinftyenergyFunctions2}
\cE_\gamma^{+}(\x_0)  :=\min_{\substack{\u \in L_{2}[0,\infty) \\ \x(0) = \x_0 , \\ \x(\infty) = \bzero}} \ \frac{1}{2} \int\displaylimits_{0}^{\infty} \Vert \y(t) \Vert^2  +  
\frac{\Vert \u(t) \Vert^2}{1-\gamma^{-2}} {\rm{d}}t.
\end{equation}	
The $\mathcal{H}_\infty$ energy functions can be computed via HJB PDEs, see Part~1 of this paper, \cite[Sec.~II]{KGB_NonlinearBT_Part1}, for more details. For LTI systems, these energy functions are quadratic in the state and can be computed by solving $\mathcal{H}_\infty$ algebraic Riccati equations. 

Under the assumption that the energy functions exist and are smooth, the open-loop nonlinear energy functions $\cE_c(\x)$ and $\cE_o(\x)$ (see \cite[Sec 2.5]{KGB_NonlinearBT_Part1} can be obtained in the limit $\gamma\rightarrow 1$ of the $\mathcal{H}_\infty$ energy functions
$\cE_\gamma^{-}(\x)$ and $\cE_\gamma^{+}(\x)$, i.e., 
\begin{equation} \label{eq:limitto1}
\lim_{\gamma\rightarrow 1} \cE_\gamma^{-}(\x) = \cE_c(\x),
\quad 
\lim_{\gamma\rightarrow 1} \cE_\gamma^{+}(\x) = \cE_o(\x).
\end{equation}

In this article, we assume similar to \cite{lukes1969optimal} that the energy functions are polynomial (or are approximated as such), which allows us scalability in our approach to the model reduction process. In particular, we assume that the past energy function is represented in the form (or is approximated as)
\begin{equation} \label{eq:pastenerexp}
    \cE_\gamma^-(\x)  \approx \frac{1}{2}\left ( \v_2^\top \kronF{\x}{2}  + \v_3^\top \kronF{\x}{3} + \cdots + \v_d^\top \kronF{\x}{d} \right ).
\end{equation}
where {$\v_k \in \Real^{n^k}$ and} $\kronF{\x}{k}$ denotes the $k$-term Kronecker product of $\x$ defined as
\begin{equation}
\kronF{\x}{k}: = \underbrace{\x \otimes \dots \otimes \x}_{k \ \text{times}} {\in \Real^{n^k}}.
\end{equation}
We also assume the future energy function $\cE_\gamma^+(\x)$ has the form (or is approximated as)
\begin{align}
\cE_\gamma^+(\x)  
& \approx \frac{1}{2}\left ( \w_2^\top \kronF{\x}{2}  + \w_3^\top \kronF{\x}{3} + \dots + \w_d^\top \kronF{\x}{d} \right ), \label{eq:wi_coeffs}
\end{align}
where {$\w_k \in \Real^{n^k}$}.
To ensure a unique representation of the coefficients in \eqref{eq:pastenerexp} and \eqref{eq:wi_coeffs}, we assume that the polynomial representations of the energy functions have a symmetric representation as explained next. 
\begin{definition}[Symmetric Coefficients\label{def:sym}] A monomial term with real coefficients $\w_d^\top \kronF{\x}{d}$ has {\em symmetric coefficients} if it satisfies
\begin{displaymath}
  \w_d^\top \hspace{-0.5ex} \left(\a_1 \otimes \a_2 \otimes \cdots \otimes \a_d\right) = \w_d^\top \hspace{-0.5ex} \left(\a_{i_1} \otimes \a_{i_2} \otimes \cdots \otimes \a_{i_d}\right),
\end{displaymath}
where the indices $\{ i_k \}_{k=1}^d$ are any  permutation of $1, \ldots, d$.
\end{definition}
Note that Algorithm~1 in Part~1 of this paper is designed to ensure symmetry in the computed coefficients.
This symmetry definition generalizes the definition of symmetry from matrices to tensors.  For example, requiring $\w_2^\top (\a\otimes \b) = \w_2^\top (\b\otimes \a)$ for any $\a$ and $\b$ is equivalent to $(\a^\top\otimes\b^\top)\w_2 = (\b^\top\otimes\a^\top)\w_2$.  Hence, using $\w_2 = \text{vec}(\W_2)$, we have $\b^\top \W_2 \a = \a^\top \W_2 \b$. Since these are real scalars, this implies $\W_2 = \W_2^\top$.

\begin{remark}
    In \cite[Sec.~3]{KGB_NonlinearBT_Part1}, we present scalable methods to compute the polynomial approximations \eqref{eq:pastenerexp} and \eqref{eq:wi_coeffs} to the energy functions. However, one may also exploit other approaches, e.g., machine learning~\cite{bouvrie2017kernel} or polynomial fitting to obtain a polynomial form of the energy functions. The methods in this paper are agnostic to how the vectors $\v_i$ and $\w_i$ are computed. 
\end{remark}

\section{From energy functions to balanced models} \label{sec:BalModels}
Energy functions are at the heart of the balancing and model reduction process, both for linear and nonlinear systems. 
In this section, we suggest a tensor-based approach to compute two different balancing transformations that efficiently ``diagonalize" these energy functions by exploiting the polynomial structure of the energy functions in \eqref{eq:pastenerexp} and \eqref{eq:wi_coeffs}. 
We compute the {\em input-normal/output-diagonal} balancing transformation in Section~\ref{sec:inputnormal}, the {\em input-output} balancing transformation in Section~\ref{sec:inputoutput}, and present the nonlinearly transformed FOM (in either input-normal/output-diagonal form or in input-output balanced form) in Section~\ref{sec:balancedFOM}.
%
\subsection{Input-normal/output-diagonal balancing transformation} \label{sec:inputnormal}
The following theorem shows the existence of a nonlinear coordinate transformation that brings the nonlinear system \eqref{eq:FOMNL1}--\eqref{eq:FOMNL2} into the input-normal/output-diagonal form.
\begin{theorem} \cite[Thm. 8]{fujimoto2010balanced} \label{thm:inputnormal}
	Suppose the Jacobian linearization of the nonlinear system is controllable, observable, and asymptotically stable. Then there is a neighborhood $\cW$ of the origin and a smooth coordinate transformation $\x = \Phi(\z)$ on $\cW$ with $\z = [z_1, z_2, \ldots, z_n]$ such that the controllability and observability energy functions have {\em input-normal/output-diagonal} form:
	\begin{align}
	\cE_c(\Phi(\z)) & = \frac{1}{2} \sum_{i=1}^n z_i^2, \label{eq:balRealCont} \\
	\cE_o(\Phi(\z)) & = \frac{1}{2} \sum_{i=1}^n  \xi_i^2(z_i) z_i^2.  \label{eq:balRealObs}
	\end{align}
\end{theorem}
\vspace{0.1cm}
The above theorem also holds for the closed-loop balancing case (even for unstable systems under additional assumptions, see \cite[Thm. 5.13]{scherpen1996hinfty_balancing}). 
In \eqref{eq:balRealObs}, the functions $\xi_i(z_i)$ are called the singular value functions of the input-normal form. One can think of \eqref{eq:balRealCont} as an energy function with equal contribution from each state to the controllability energy.  Unlike the linear case, the singular value functions $\xi_i(z_i)$ are not constant and are state-dependent. Notably, the $i$th singular value function $\xi_i$ only depends on the $i$th state $z_i$, which allows for truncation of the individual states when constructing the ROM, as we see in Section~\ref{sec:ROMs}.

\subsubsection{Taylor expansion of singular value functions and the state transformation}
Under the conditions of Theorem \ref{thm:inputnormal}, a $C^\infty(\Real^n)$ mapping $\x = \Phi(\z)$ exists to transform the nonlinear system into the input-normal/output-diagonal form. Here, we assume that this transformation is analytic, so that we can write (or approximate) it as
\begin{align} 
\x  = \Phi(\z) & = \T_1 \z + \T_2 \kronF{\z}{2} + \dots + \T_{k} \kronF{\z}{k} \nonumber \\
& =: \T_1 \z + \Phi^h(\z) \label{eq:trafo},
\end{align}
where $\T_j \in \Real^{n\times n^j}$ are the polynomial coefficients, $\T_1$ is nonsingular, and $\kronF{\z}{k}$ is, as before, the $k$-times Kronecker product of $\z$. This allows for scalable computation of this nonlinear transformation. To abbreviate notation in what follows we group the higher-degree terms in \eqref{eq:trafo} into the term $\Phi^h$. 
We similarly expand the singular value functions of the input-normal/output-diagonal form associated with the $i$th state component as a degree $\ell$ polynomial
\begin{align} 
    \xi_i(z_i) 
    & = \xi_i(0) + c_i^{(1)} z_i + c_i^{(2)} z_i^2 + \cdots + c_i^{(\ell)} z_i^\ell \nonumber\\
    & =: \xi_i(0) + \xi_i^h(z_i),  \label{eq:expandSval}
\end{align}
for $i=1, 2, \ldots, n$. Define the coefficients of the degree $j$th terms as $\c_j= [c_1^{(j)}, \ c_2^{(j)}, \ldots, c_n^{(j)}]^\top$, which allows us to write the vector of singular value functions $\mathbf{\xi}(\z) = \left [ \begin{matrix} \xi_1(z_1) & \xi_2(z_2) & \cdots & \xi_n(z_n)  \end{matrix} \right ]^\top$ of the input-normal/output-diagonal form as
\begin{align} \label{eq:sval_expansion}
    \mathbf{\xi}(\z)  =  \bXi \cdot \mathbf{1} + \diag(\c_1)\z + \dots + \diag(\c_\ell) \z^\ell,
\end{align}
where $\z^\ell$ denotes  \emph{the componentwise power} of the vector $\z$, the diagonal matrix $\bXi= {\normalfont \text{diag}} (\xi_1^2(0), \ldots, \xi_n^2(0))$, and $\mathbf{1}$ is the vector of ones. The following two sections derive tensor expressions for the efficient computation of $\T_i$ and $\c_i$. 

\subsubsection{Computing the polynomial coefficient matrices of the input-normal/output-diagonal nonlinear transformation}
This section outlines the computations required to obtain the polynomial coefficient matrices $\{\T_i\}_{i=1}^{k}$ for the nonlinear transformation \eqref{eq:trafo}. 
To begin, we rewrite \eqref{eq:pastenerexp} and \eqref{eq:wi_coeffs} as
\begin{align} 
\cE_\gamma^-(\x) & = \frac{1}{2}\x^\top \V_2\x  + \cE_c^h(\x), \label{eq:NLen}\\
\cE_\gamma^+(\x) & = \frac{1}{2}\x^\top \W_2\x + \cE_o^h(\x),\label{eq:NLen2}
\end{align}
with $\cE_c^h(\x) = \frac{1}{2}\left ( \v_3^\top \kronF{\x}{3} + \cdots + \v_d^\top \kronF{\x}{d} \right)$ and $\cE_o^h(\x) = \frac{1}{2}\left ( \w_3^\top \kronF{\x}{3} + \cdots + \w_d^\top \kronF{\x}{d} \right)$.
Note, that from \cite[Thm 5.17]{scherpen1996hinfty_balancing} it follows that $\text{unvec}(\v_2) = \V_2 = \Y_\infty^{-1}$ and $\text{unvec}(\w_2) = \W_2 =\X_\infty$ are the unique symmetric positive-definite stabilizing solutions to the $\mathcal{H}_\infty$ filter ARE
{\small 
\begin{equation}
    \A \Y_\infty + \Y_\infty \A^\top + \B \B^\top - (1-\gamma^{-2}) \Y_\infty \C^\top \C \Y_\infty = \bzero, \label{eq:HinftyRiccati1}
\end{equation}
}    
and the $\cH_\infty$ control ARE  
{\small
\begin{equation}
    \A^\top \X_\infty + \X_\infty \A + \C^\top \C - (1-\gamma^{-2}) \X_\infty \B \B^\top \X_\infty = \bzero. \label{eq:HinftyRiccati2}
\end{equation}
}

\noindent
To obtain a balanced representation, we insert the transformation \eqref{eq:trafo} into the past energy function~\eqref{eq:NLen} and enforce the input-diagonal structure in equation~\eqref{eq:balRealCont} to obtain
\begin{align}
\z^\top \z = & \z^\top \T_1^\top \V_2 \T_1 \z + 2\z^\top \T_1^\top \V_2 \Phi^h(\z) \nonumber \\
& + \Phi^h(\z)^\top \V_2 \Phi^h(\z) + 2 \cE_c^h(\Phi(\z)). \label{eq:phiMatch1}
\end{align}
We can then proceed to compute the matrices $\{\T_i\}_{i=1}^{k}$ by matching polynomial coefficients. 
Similarly, inserting the transformation \eqref{eq:trafo} into the future energy function \eqref{eq:NLen2} and enforcing diagonalization as in  \eqref{eq:balRealObs} yields
\begin{align}
 \sum_{i=1}^n &  z_i^2 \left ( \xi_i^2(0) + 2\xi_i(0) \xi_i^h(z_i) + \xi_i^h(z_i)^2 \right ) \nonumber \\
& = \z^\top \T_1^\top \W_2 \T_1 \z + 2\z^\top \T_1^\top \W_2 \Phi^h(\z) \nonumber \\
& \quad + \Phi^h(\z)^\top \W_2 \Phi^h(\z) + 2 \cE_o^h(\Phi(\z)).
\label{eq:phiMatch2}
\end{align}
Knowing the $\{\T_i\}_{i=1}^{k}$ from the previous step, we can solve \eqref{eq:phiMatch2} for the singular value functions.

Before stating the next theorem, we introduce the notation
\begin{equation} \label{eq:tensorproducts}
\mathcal{T}_{m,l} = \sum_{\sum{i_j}=l} \T_{i_1}\otimes \cdots \otimes \T_{i_m} \in \Real^{n^m\times n^l},
\end{equation}
where $i_j \geq 1$ for each $j=1,\ldots,m$. Thus, $\mathcal{T}_{m,l}$ denotes all unique tensor products with $m$ terms and $n^{l}$ columns. For instance, $\mathcal{T}_{2,3}= \T_1\otimes \T_2 + \T_2\otimes\T_1$, or $\mathcal{T}_{3,4}= \T_1\otimes\T_1\otimes \T_2 + \T_1\otimes\T_2\otimes \T_1 +\T_2\otimes\T_1\otimes \T_1$, which contains three unique triple tensor products, and $\mathcal{T}_{4,4} = \kronF{\T_1}{4}$.
\begin{theorem}[Input-normal/output-diagonal transform] \label{thm:polyCoeffsTrafo}
   Let $\w_i,\v_i$ be the vectors of polynomial coefficients for the energy functions in \eqref{eq:pastenerexp} and \eqref{eq:wi_coeffs}. Let $\L, \R$ be Cholesky factors of 
   $\W_2 = {\rm unvec}(\w_2)$ and $\V_2={\rm unvec}(\v_2)$, i.e., $\W_2 = \L \L^\top$ and $\V_2 = \R \R^\top$.  Compute the singular value decomposition of $\L^\top \R^{-\top} = \mathcal{U} \bXi \mathcal{V}^\top$. 
    The linear transformation $\T_1$ in \eqref{eq:trafo} and its inverse $\T_1^{-1}$ are given by
      \begin{equation}
        \T_1 = \R^{-\top} \mathcal{V} \in \Real^{n\times n}, \quad {\T_1^{-1}} = \bXi^{-1} \mathcal{U}^\top \L^\top \in \Real^{n\times n}
    \end{equation}
and they satisfy $\T_1^\top \V_2 \T_1 = \I$ (input-normal) and $\T_1^\top \W_2 \T_1 = \bXi^2$ (output-diagonal). Moreover, $\T_1^{-1} \V_2^{-1} \W_2 \T_1=\bXi^2$, i.e., $\V_2^{-1}\W_2$ is similar to $\bSigma^2$. Thus, $\T_1$ is the well-known input-normal/output-diagonal transformation for the linearized system. 
The higher-degree tensors of the nonlinear transformation \eqref{eq:trafo} for $k\geq 2$ take the form
{\small
\begin{subequations}
\begin{align}
    \T_{k} &= -\frac{1}{2} \T_1 {\rm{unvec}}\left( \M_k\right)^\top \in \Real^{n\times n^k},~\mbox{where} \label{eq:Tkformula1} \\
   \M_k & =  \sum_{\substack{i,j>1\\ i+j=k+1}}  {\rm{vec}} \left ( \T_{j}^\top \V_2 \T_{i} \right ) +     \sum_{i=3}^{k+1} \cT_{i,k+1}^\top \v_i  \label{eq:Tkformula2} \end{align}
\end{subequations}
}
\end{theorem}
\vspace{0.2cm}

\begin{remark} \label{remark:kroneckerPoly}
Throughout, we often use the fact that the symmetry of the matrix $\V_2$ leads to $\z^\top \T_1^\top \V_2 \Phi^h(\z) + \Phi^h(\z)^\top \V_2 \T_1 \z  = 2\z^\top \T_1^\top \V_2 \Phi^h(\z)$ since these are scalar quantities. Likewise $(\kronF{\z}{k})^\top \T_k^\top \V_2 \T_1 \z  = \z^\top \T_1^\top \V_2 \T_k \kronF{\z}{k}$ leads to $\text{vec}(\T_1^\top \V_2 \T_k)^\top \kronF{\z}{k+1} = \text{vec}(\T_k^\top \V_2 \T_1)^\top \kronF{\z}{k+1}$. In general, $\text{vec}(\T_1^\top \V_2 \T_k)^\top \b \neq \text{vec}(\T_k^\top \V_2 \T_1)^\top \b$ for an arbitrary vector $\b$.
\end{remark}

The tensors $\T_i$ for the quadratic and cubic part of the transformation have the specific form
\begin{align} \label{eq:T2}
     \T_2  =& -\frac{1}{2} \T_1 \ {\rm{unvec}}( [\kronF{\T_1}{3}]^\top \v_3)^\top\\
    \T_3 =& -\frac{1}{2} \T_1 \ {\rm{unvec}}\left (  {\rm{vec}}(\T_{2}^\top \V_2 \T_{2} ) \ldots \right. \nonumber \\
    & \left.  \qquad \qquad \qquad+   \cT_{3,4}^\top \v_3 + [\kronF{\T_1}{4}]^\top \v_4  \right )^\top.
    \label{eq:T3}
\end{align}

\begin{proof}
We start by proving the results for the matrix $\T_1$. Comparing the quadratic terms in~\eqref{eq:phiMatch1}, we conclude that $\T_1^\top \V_2 \T_1 = \I$ and then comparing quadratic terms in~\eqref{eq:phiMatch2} we obtain $\T_1^\top \W_2 \T_1 = \text{diag}(\xi_1^2(0), \ldots, \xi_n^2(0))=\bXi^2$. Taken together these yield $\bXi^2 = \I^{-1} \T_1^\top \W_2 \T_1 = \T_1^{-1} \V_2^{-1} \T_1^{-\top} \T_1^\top \W_2 \T_1 = \T_1^{-1} \V_2^{-1} \W_2 \T_1$. This shows that $\V_2^{-1}\W_2$ is similar to $\bXi^2$, i.e., they both have eigenvalues $\xi_i^{2}(0)$.  Therefore $\T_1$ is the input-normal/output-diagonal linear balancing transformation that uses the $\cH_\infty$-ARE solutions, see~\eqref{eq:HinftyRiccati1}--\eqref{eq:HinftyRiccati2}, where $\V_2 = \Y_\infty^{-1}$ and $\W_2 =\X_\infty$.

We next determine $\T_2, \ldots, \T_k$ by observing that the cubic and higher-degree terms  of~\eqref{eq:phiMatch1} are zero, i.e.,  
\begin{align}
0 & =2\z^\top \T_1^\top \V_2 \Phi^h(\z) + \Phi^h(\z)^\top \V_2 \Phi^h(\z) + 2 \cE_c^h(\Phi(\z)).
\end{align}
Recall that $\Phi^h(\z) = \T_2 \kronF{\z}{2} + \dots + \T_k \kronF{\z}{k}$; thus  for $k\geq3$ we match the degree $k$th terms to obtain
\begin{align}
 - 2& \z^\top \T_1^\top \V_2 \T_{k-1} \kronF{\z}{k-1} \nonumber \\
& =  \left [ \Phi^h(\z)^\top \V_2 \Phi^h(\z)\right ]^{\{ k\}}  + 2 [ \cE_c^h(\Phi(\z))]^{\{ k\}}, \label{eq:phiMatchTerms1}
\end{align}%
where $[\cdot]^{\{ k \}}$ selects the degree $k$ terms of the expressions inside the bracket. 
Recall that the higher-degree terms in the energy functions start with cubic contributions, i.e., $\cE_c^h(\x) = \frac{1}{2}\left ( \v_3^\top \kronF{\x}{3} + \cdots + \v_d^\top \kronF{\x}{d} \right )$.
We next use that for a given matrix $\M$ we have  $\z^\top \M \z = \text{vec}(\M^\top)^\top (\z \otimes \z)$, see also Remark~\ref{remark:kroneckerPoly}. Consequently, $\z^\top \T_1^\top \V_2 \T_{k-1} \kronF{\z}{k-1} = \text{vec}(\T_{k-1}^\top \V_2 \T_1)^\top \kronF{\z}{k}$. 
For general $k$, equation \eqref{eq:phiMatchTerms1} can therefore be rewritten as
\begin{align}
 -2 & \text{vec}(\T_{k-1}^\top \V_2 \T_1)^\top \kronF{\z}{k} \\
& = \sum_{\substack{i,j>1\\ i+j=k}}  [\kronF{\z}{i}]^\top \T_{i}^\top \V_2 \T_{j} \kronF{\z}{j}  +  \left (   \sum_{i=3}^k \v_i^\top \kronF{[\Phi(\z)]}{i} \right )^{\{ k\}}, \nonumber
\end{align}
since $\V_2$ is symmetric. Expanding $\Phi(\z)$ and using again $\z^\top \M \z = \text{vec}(\M^\top)^\top (\z \otimes \z)$ yields
{\small 
\begin{align} 
 -2 & \text{vec}(\T_{k-1}^\top \V_2 \T_1)^\top \kronF{\z}{k} \nonumber \\
     & =   \sum_{\substack{i,j>1\\ i+j=k}}  \text{vec} \left ( \T_{j}^\top \V_2 \T_{i} \right )^\top \kronF{\z}{k}
     \label{Tkhigher} \nonumber\\
    & +  \left (   \sum_{i=3}^k \v_i^\top \kronF{[\T_1\z + \T_2 \kronF{\z}{2} +\dots + \T_{k}\kronF{\z}{k}]}{i}\right )^{\{ k\}}. 
\end{align}
}

\noindent 
We next use $\V_2 \T_1 = \T_1^{-\top}$ and the definition of the unique tensor products with $i$ terms and $n^{k}$ columns, $ \cT_{i,k}$ and observe that \eqref{Tkhigher} holds if
\begin{align} 
    &-2 \text{vec}(\T_{k-1}^\top \T_1^{-\top})^\top \nonumber \\
    & =   \sum_{\substack{i,j>1\\ i+j=k}}  \text{vec} \left ( \T_{j}^\top \V_2 \T_{i} \right )^\top +    \sum_{i=3}^k \v_i^\top \cT_{i,k} .\label{eq:phiMatchTerms1a}
\end{align}

Reindexing from $k-1$ to $k$ and reshaping \eqref{eq:phiMatchTerms1a} proves
\eqref{eq:Tkformula1}--\eqref{eq:Tkformula2}.
Specifically, for $k=2$, the first sum in $\M_k$ in~\eqref{eq:Tkformula2} is zero,
to produce~\eqref{eq:T2}.
For $k=3$, equation~\eqref{eq:Tkformula2} yields~\eqref{eq:T3}.
\end{proof}

\noindent 
The next section focuses on computing the singular value functions
$\xi_i^h(z_i)$
associated with the input-normal form, which we obtain by polynomial expansion. 

\subsubsection{Computing the state-dependent singular value functions}
With the $\{\v_i\}_{i=2}^d$ and $\{\w_i\}_{i=2}^d$ available from the energy functions \eqref{eq:pastenerexp} and \eqref{eq:wi_coeffs}, and the $\{\T_i\}_{i=1}^{k}$ from Theorem~\ref{thm:polyCoeffsTrafo}, it remains to compute the polynomial coefficients $\{\c_i\}_{i=1}^{\ell}$ of the singular value functions defined in~\eqref{eq:expandSval}--\eqref{eq:sval_expansion}. The singular value functions are used to determine which modes to keep in the balanced ROM. To determine the singular value functions, consider the cubic and higher-degree terms of equation \eqref{eq:phiMatch2}, i.e., 
\begin{align} \label{eq:svfunctionsCoeff}
&\sum_{i=1}^n z_i^2 \xi_i^h(z_i) \left (2\xi_i(0) + \xi_i^h(z_i) \right )  \\
& =  2\z^\top \T_1^\top \W_2 \Phi^h(\z) + \Phi^h(\z)^\top \W_2 \Phi^h(\z) + 2 \cE_o^h(\Phi(\z)), \nonumber 
\end{align}
for which we know the terms on right-hand side ($\T_i$ from Theorem~\ref{thm:polyCoeffsTrafo} and $\w_i$ from \eqref{eq:wi_coeffs}) as well as $\xi_i(0)$ in the left-hand side from Theorem~\ref{thm:polyCoeffsTrafo}. 
The following theorem shows how to  compute the coefficients $\c_i$ associated with the state-dependent part of the singular value functions in~\eqref{eq:expandSval}. 
\begin{theorem}
    Let $\z = [z_1, z_2, \ldots, z_n]^\top$ be the transformed state and $\c_k= [c_1^{(k)}, \ c_2^{(k)}, \ldots, c_n^{(k)}]^\top$ be the $n$ dimensional unknown coefficients vector for the degree $k$ terms as defined in~\eqref{eq:expandSval} and~\eqref{eq:sval_expansion}. The coefficients for the terms that are linear ($k=1$) in the state can be obtained as  
    \begin{equation}  \label{eq:c1formula}
        \c_1  = \bXi^{-1} \left ({\normalfont \text{vec}} (\T_{2}^\top \W_2 \T_1)^\top  + \frac{1}{2}\w_3^\top \kronF{\T_1}{3} \right)_{\mathcal{I}_1}
    \end{equation}
    for the indices $\mathcal{I}_1=\{ j \ | \ j=(i-1)(n^2+n) +i, \ i=1,\ldots, n\}$. In general, for $k\geq 1$ we obtain the explicit recursion
    \begin{align} \label{eq:generalck}
    \c_{k} = & \frac{1}{2} \bXi^{-1} \left [ \left(\C_k\right)_{\mathcal{I}_{k}} - \sum_{i+j=k} \c_i\odot\c_j  \right ]
    \end{align}
    where
    \begin{align}  \label{eq:bigCk}
        \C_{k} = \hspace*{-1ex} \sum_{\substack{i,j\geq 1\\ i+j=k+2}}  \hspace*{-1ex}  {\normalfont \text{vec}} \left ( \T_{j}^\top \W_2 \T_{i} \right )^\top   + \sum_{i=3}^{k+2} \left(\cT_{i,k+2}\right)^\top\w_i,   
    \end{align} 
 $\mathcal{I}_{k}$ is the index set $\mathcal{I}_{k}=\{j \ | \ j=(i-1)\sum_{l=1}^{k+1} n^l+i, \ i=1,\ldots, n \}$, and $\odot$ denotes the Hadamard product (componentwise multiplication).
\end{theorem}
\begin{proof}
We start by matching degree $k$ polynomial terms on both sides of \eqref{eq:svfunctionsCoeff} to obtain
\begin{subequations} \label{eq:phiMatchTerms2}
\begin{align} \label{eq:phiMatchTerms2a}
&\sum_{i=1}^n z_i^2 \left [ \xi_i^h(z_i) \left (2\xi_i(0) + \xi_i^h(z_i) \right ) \right ]^{\{ k-2\}}  \\
& = \sum_{i=1}^n z_i^2 \left [ \left(c_i^{(1)} z_i + c_i^{(2)} z_i^2 + \dots + c_i^{(\ell)} z_i^\ell\right) \right. \label{eq:phiMatchTerms2b} \\
& \left. \qquad \cdot \left (2\xi_i(0) +c_i^{(1)} z_i + c_i^{(2)} z_i^2 + \dots + c_i^{(\ell)} z_i^\ell \right ) \right ]^{\{ k-2\}} \nonumber   \\ 
& = 2\z^\top \T_1^\top \W_2 \T_{k-1} \kronF{\z}{k-1} + [ \Phi^h(\z)^\top \W_2 \Phi^h(\z)]^{\{ k\}} \nonumber \\
& \ \ \ + 2 [\cE_o^h(\Phi(\z))]^{\{ k\}}.  \label{eq:phiMatchTerms2c}
\end{align}
\end{subequations}
Similar to the proof of Theorem~\ref{thm:polyCoeffsTrafo}, for $k\geq 3$ we obtain  the system
\begin{align}
&\sum_{i=1}^n z_i^2 \left [ (c_i^{(1)} z_i + c_i^{(2)} z_i^2 + \dots + c_{i}^{(k-2)} z_i^{(k-2)} ) \right. \cdot \nonumber \\
& \left. \left (2\xi_i(0) +c_i^{(1)} z_i + c_i^{(2)} z_i^2 + \dots + c_{i}^{(k-3)} z_i^{(k-3)} \right ) \right ]^{\{ k-2\}} \nonumber \\
& = \left ( \sum_{\substack{i,j\geq 1\\ i+j=k}}  \text{vec} \left ( \T_{j}^\top \W_2 \T_{i} \right )^\top +    \sum_{i=3}^k \w_i^\top \cT_{i,k}  \right ) \kronF{\z}{k}, 
\label{eq:ci_equality}
\end{align}
where in~\eqref{eq:phiMatchTerms2b} we replaced the first $\ell$ with $k-2$ and the second $\ell$ with $k-3$, as those correspond to the highest-degree polynomials we expect to match inside the $\{k-2\}$ selection. 
Moreover, we exploited the symmetry of $\W_2$ and the fact that $2{\normalfont \text{vec}} (\T_{k+1}^\top \W_2 \T_1)^\top \kronF{\z}{2} = {\normalfont \text{vec}} (\T_{1}^\top \W_2 \T_{k+1})^\top \kronF{\z}{2}+ {\normalfont \text{vec}} (\T_{k+1}^\top \W_2 \T_1)^\top \kronF{\z}{2}$.

Focusing on the case $k=3$ then yields an equation for the coefficients $\c_1$:
\begin{align}  \label{eq:sumci}
2 \sum_{i=1}^n \xi_i(0) c_i^{(1)}   z_i^3  = \left ( 2 \text{vec}(\T_{2}^\top \W_2 \T_1)^\top  + \w_3^\top \kronF{\T_1}{3} \right )\kronF{\z}{3}. 
\end{align}

For $i=1,2,\ldots,n$, matching the coefficients of $z_i^3$ on both sides of~\eqref{eq:sumci} yields~\eqref{eq:c1formula}, the formula for
$\c_1$, where the index set $\mathcal{I}_1$ corresponds to the location of the monomials in $\kronF{\z}{3}$.
Next, we focus on quartic ($k=4$) terms in~\eqref{eq:ci_equality} and obtain
\begin{align}
& \sum_{i=1}^n  \left (2c_i^{(2)}\xi_i(0) + (c_i^{(1)})^2  \right )z_i^4 \nonumber \\
& =  \left ( 2 \text{vec}(\T_{3}^\top \W_2 \T_1)^\top  + {\rm{vec}}(\T_2^\top \W_2 \T_2)^\top \right. \nonumber\\
& \left. \quad + \w_3^\top \cT_{3,4} + \w_4^\top \kronF{\T_1}{4} \right ) \kronF{\z}{4}.
\end{align}
Again, for $i=1,2,\ldots,n$ we match the coefficients of $z_i$ on both sides and so for $\c_2 = [c_1^{(2)}, \ c_2^{(2)}, \ldots, c_n^{(2)}]^\top$, we obtain the formula 
\begin{align}
 \c_2  = & \frac{1}{2} \bXi^{-1} \Big [ \Big (2 \text{vec}(\T_{3}^\top \W_2 \T_1)  + {\rm{vec}}(\T_2^\top \W_2 \T_2) \nonumber \\
 & + \cT_{3,4}^\top \w_3 + (\kronF{\T_1}{4})^\top \w_4 \Big )_{\mathcal{I}_2} - \c_1^2 \Big ]
\end{align}
for the indices $\mathcal{I}_2= \{ j \ | \ j=(i-1)(n^3 + n^2 +n) + i, \ i=1,\ldots, n\}$, which corresponds to the location of the monomials in $\kronF{\z}{4}$ and where $\c_1^2$ denotes the componentwise square of $\c_1$. 
The case $k=5$ yields
\begin{align}
\sum_{i=1}^n & \left (2 c_i^{(1)} c_i^{(2)} + 2\xi_i(0)c_i^{(3)} \right ) z_i^5 \nonumber  \\
 = & \Big (  2\text{vec} (\T_{4}^\top \W_2 \T_1)^\top + 2 \text{vec} \left (\T_{2}^\top \W_2^\top \T_{3}\right )^\top \nonumber \\
& + \sum_{i=3}^5 \w_i^\top \cT_{i,k}  \Big ) \kronF{\z}{k},
\end{align}
leading to the explicit expression for $\c_3$ as
\begin{align}
 \c_3 = & \frac{1}{2} \bXi^{-1} \Big [\Big (2\text{vec} (\T_{4}^\top \W_2 \T_1) \nonumber \\
& + 2 \text{vec} \Big ( \T_{3}^\top \W_2^\top \T_{2}  \Big ) \nonumber \\
& +    \sum_{i=3}^5 \cT_{i,k}^\top \w_i \Big )_{\mathcal{I}_3} -2\c_1\odot\c_2  \Big ] 
\end{align}
for the indices $\mathcal{I}_3 = \{ j \ | \ j=(i-1)(n^4+ n^3 + n^2+n) +i, \ i=1, \ldots, n \}$. 
Continuing in this fashion  yields the recursion formula~\eqref{eq:generalck} for $\c_k$ for general $k$ after the index change $k-2 \to k$.
\end{proof}
%
%

\subsubsection{Complete algorithm and implementation}
A complete algorithm for the tensor-based computation of the nonlinear input-normal/output-diagonal transformation $\Phi(\z)$ is given in Algorithm~\ref{algo:inputnormalTrafo} and the associated singular value functions $\xi_i(z_i)$ in Algorithm~\ref{algo:inputnormalSvals}.
    
\begin{algorithm}[htb]
	\caption{Computation of polynomial coefficients $\{ \T_i \}_{i=1}^k$ in approximations to $\Phi(\z)$ from \eqref{eq:trafo}.}
	\label{algo:inputnormalTrafo}
	\begin{algorithmic}[1]
		\Require Coefficients $\{\v_i\}_{i=2}^d$ and $\{\w_i\}_{i=2}^d$  from \eqref{eq:pastenerexp} and \eqref{eq:wi_coeffs}. 
		\Ensure Coefficients $\{\T_i\}_{i=1}^{k}$ of the transformation and $\bXi = \diag(\xi_1(0), \ldots, \xi_n(0))$
		\State Compute Cholesky factors $\V_2 = \R \R^\top$ and $\W_2 = \L \L^\top$
		\State Compute the singular value decomposition $\L^\top \R^{-\top} = \mathcal{U} \bXi \mathcal{V}^\top$ and set 
        \begin{equation*}
            \T_1 = \R^{-\top} \mathcal{V} 
        \end{equation*}
        \State Compute the higher-degree polynomial coefficients:
        {\footnotesize{
        \begin{align*}
             \T_2 & = -\frac{1}{2} \T_1 \ {\rm{unvec}}( [\kronF{\T_1}{3}]^\top \v_3)^\top\\
            \T_3 &= -\frac{1}{2} \T_1 \ {\rm{unvec}}\left (  {\rm{vec}}(\T_{2}^\top \V_2 \T_{2} ) +  \cT_{3,4}^\top \v_3 + [\kronF{\T_1}{4}]^\top \v_4  \right )^\top\\
            \T_{k} &= -\frac{1}{2} \T_1 {\rm{unvec}} \left ( \sum_{\substack{i,j>1\\ i+j=k+1}}  {\rm{vec}} \left ( \T_{j}^\top \V_2 \T_{i} \right ) +     \sum_{i=3}^{k+1} \cT_{i,k+1}^\top \v_i \right )^\top
        \end{align*}
        }}
        \State Symmetrize the coefficients $\T_{k}$ at each step (see Remark~\ref{remark:kroneckerPoly}).
    \end{algorithmic}
\end{algorithm}

\begin{algorithm}[htb]
	\caption{Computation of the input-normal singular value functions $\xi_i(z_i)$ from \eqref{eq:sval_expansion}.}
	\label{algo:inputnormalSvals}
	\begin{algorithmic}[1]
		\Require Coefficients $\{\v_i\}_{i=2}^d$ and $\{\w_i\}_{i=2}^d$  from  \eqref{eq:pastenerexp} and \eqref{eq:wi_coeffs}, $\{\T_i\}_{i=1}^k$ and $\bXi$ from Algorithm~\ref{algo:inputnormalTrafo}. 
		\Ensure Coefficients of singular value functions $\{ \c_i\}_{i=1}^\ell$.
        \State Compute coefficients for the linear terms 
            \begin{equation*}
                \c_1  = \bXi^{-1} \left ({\normalfont \text{vec}} (\T_{2}^\top \W_2 \T_1)  + \frac{1}{2} (\kronF{\T_1}{3})^\top \w_3 \right)_{\mathcal{I}_1} 
            \end{equation*}
        for the indices ${\mathcal{I}_1}=\{ j \ | \ j=(i-1)(n^2+n) +i, \ i=1, \ldots, n\}$. 
        \State For $m=1,2,\ldots, \ell$, compute the coefficients
             \begin{align*}
                 \c_{m} &= \frac{1}{2} \bXi^{-1} \left [  \left (\sum_{\substack{i,j\geq 1\\ i+j=m+2}}  {\normalfont \text{vec}} \left ( \T_{j}^\top \W_2^\top \T_{i} \right ) \right. \right. \nonumber \\ 
                & +  \left. \left.   \sum_{i=3}^{m+2} \cT_{i,m+2}^\top \w_i \right)_{\mathcal{I}_{m}} - \sum_{i+j=m} \c_i\odot\c_j  \right ] 
            \end{align*}
            for the index set $\mathcal{I}_{m}=\{j \ | \ j=(i-1)\sum_{l=1}^{m+1} n^l+i, \ i=1,\ldots, n \}.$
	\end{algorithmic}
\end{algorithm}

\begin{remark}
In the linear case the energy functions are quadratic, e.g., 
$\cE_\gamma^-(\x) = \frac{1}{2}\v_2^\top\kronF{\z}{2}$, and hence $\v_i=\bzero$ for $i\geq3$. We see from Algorithm~\ref{algo:inputnormalTrafo} that $\T_i = \bzero$ for $i\geq 2$. Thus, we recover the usual linear state transformation $\Phi(\z) = \T_1\z$. Moreover, we know that the singular value functions are constant, and we see from Algorithm~\ref{algo:inputnormalSvals} that indeed $\c_i=\bzero$ for $i\geq 1$. In sum, for the linear case, the energy functions are quadratic, the transformation linear, and the singular value functions constant. However, this cascade of degrees does not hold for the general nonlinear case. Assume the energy function is exactly cubic, i.e., $\cE_\gamma^-(\x) =\frac{1}{2}(\v_2^\top \kronF{\z}{2} + \v_3^\top \kronF{\z}{3})$. In Algorithm~\ref{algo:inputnormalTrafo} we are still able to compute $\T_k, k\geq3$ as $\T_3\neq \bzero$ and consequently the first sum on the right-hand-side of the $\T_k$ computation is also nonzero. We similarly see that the $\c_i$ coefficients in Algorithm~\ref{algo:inputnormalSvals} can be nonzero. Thus the degree of the energy function has, in general, no direct impact on the degree of the transformation and singular value functions.
\end{remark}

\subsubsection{Demonstration}
We demonstrate the effectiveness of using higher degree transformations to bring energy functions into input-normal/output-diagonal form. Using the example found in \cite[Sec. IV.B]{KGB_NonlinearBT_Part1} (modified from \cite{kawano2016model}), we consider
\begin{equation}
\label{eq:quadratic}
    \dot{\x} = \A \x + \N (\x \otimes \x) + \B u , \qquad y = \C\x 
\end{equation}
where 
\begin{align}
\label{eq:example1}
\begin{split}
\A\x & = \begin{bmatrix} -x_1 + x_2 \\ -x_2 \end{bmatrix}, \quad \N(\x\otimes \x) = \begin{bmatrix} -x_2^2 \\ 0\end{bmatrix}, \\  
\B & = \begin{bmatrix} 1 \\ 1 \end{bmatrix}, \quad \C  = [1 \ 1].
\end{split}
\end{align}
A plot of a degree 8 approximation to the past energy function, ${\cal E}_\gamma^-(\x)$ with $\eta=0.1$ (or $\gamma=\sqrt{2}$), is provided in Fig.~\ref{fig:pastEnergyOrig}. It is clear that this function is not quadratic.  To bring it into input-normal form (in which the past energy function is quadratic, see~\eqref{eq:balRealCont}), we qualitatively compare using a linear transformation $\x=\T_1\z$ (Fig.~\ref{fig:pastEnergyLinearTrans}) and a quadratic transformation $\x=\T_1\z+\T_2\kronF{\z}{2}$ (Fig.~\ref{fig:pastEnergyQuadraticTrans}).  While the linear state transformation does not adequately transform the energy into quadratic form, we observe a better local approximation to \eqref{eq:balRealCont} using the quadratic state transformation. The quality of this transformation extends over the region $(-0.2,0.2)\times (-0.2,0.2)$ in the $z$-coordinates.  

To quantitatively assess the ability of polynomial transformations to place the past energy function into input-normal form, i.e., to make it quadratic in the state, we tabulate the maximum error between the transformed energy function and the desired form \eqref{eq:balRealCont} over regions of different sizes in Table~\ref{tab:TransformedPastEnergy}.  As expected, higher degree transformations provide a better solution to the balancing problem in regions close to the origin.  We also observe the benefits of using higher degree transformations as we move closer to the origin.  These transformations are developed as local approximations, yet we see improvements over the linear transformation in a larger region.  Furthermore, the accuracy away from the origin can be achieved with a quadratic or cubic transformation in this example.  As we move closer to the origin, the superiority of the higher degree transformations arises. 

\begin{figure}
    \centering
    \includegraphics[width=0.475\textwidth]{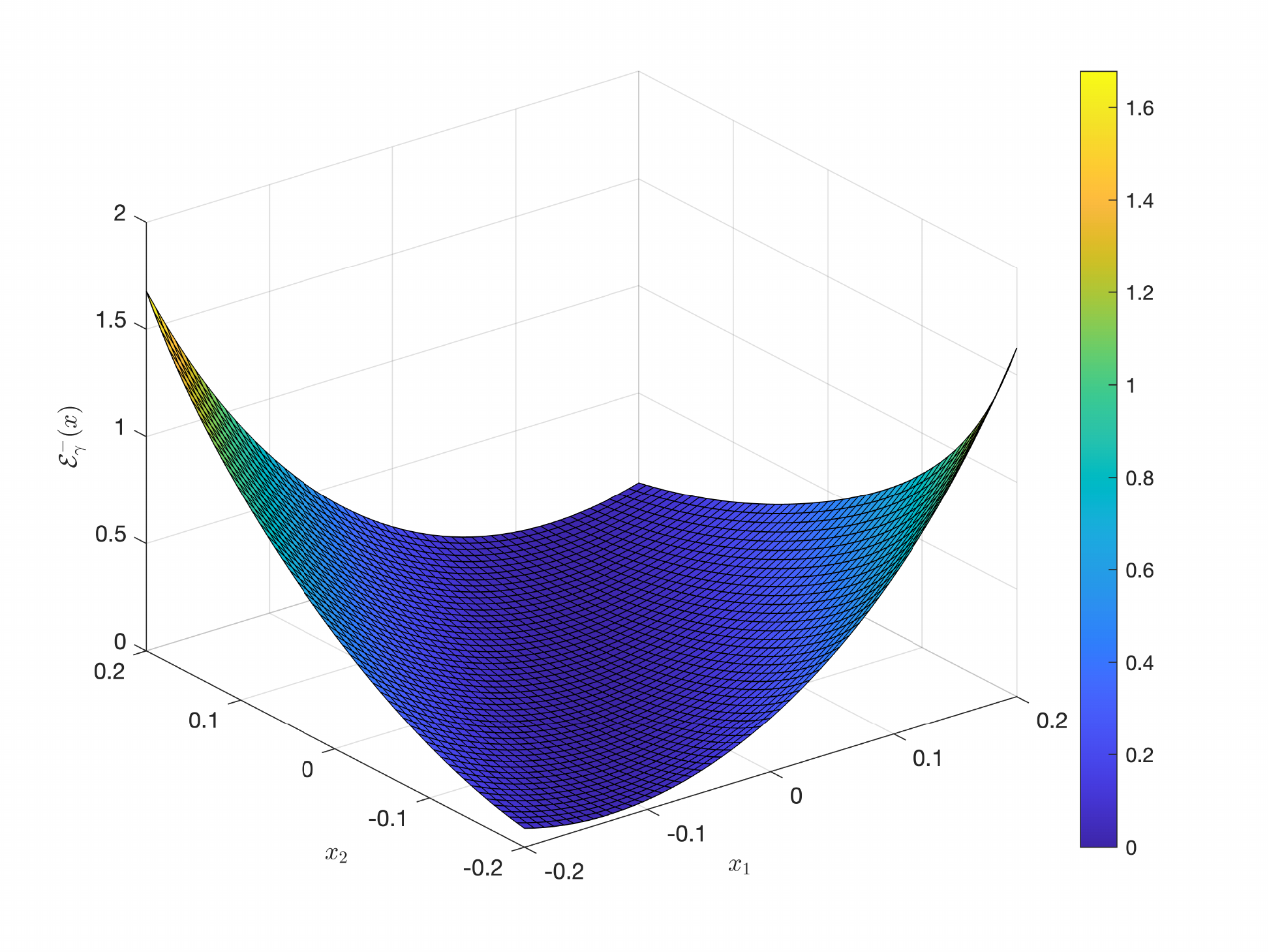}
    \caption{\label{fig:pastEnergyOrig}Past energy function $\mathcal{E}_\gamma^-$ for the model \eqref{eq:quadratic}--\eqref{eq:example1} in the original coordinates.}
\end{figure}
\begin{figure}
    \centering
    \includegraphics[width=0.475\textwidth]{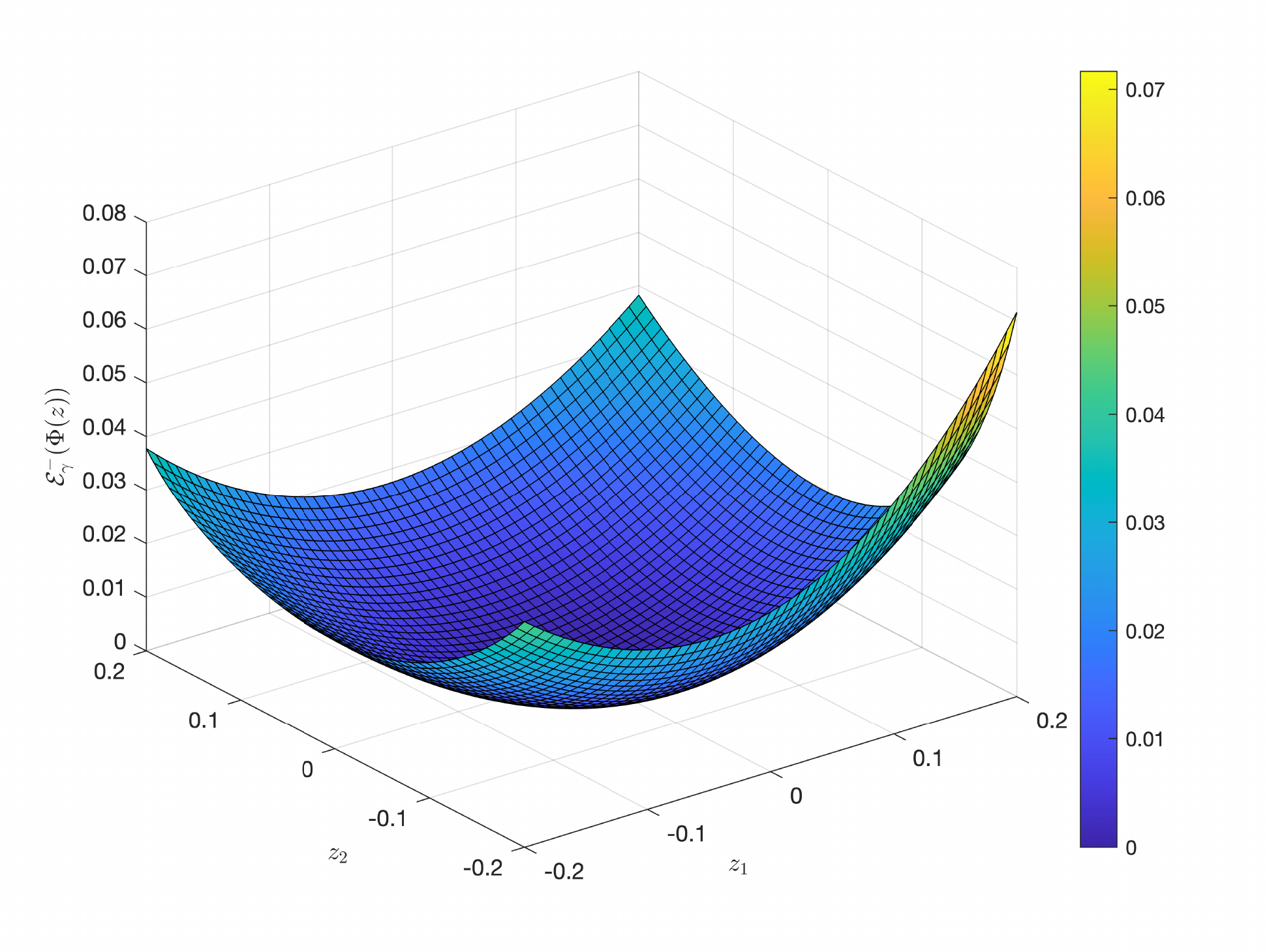}
    \caption{\label{fig:pastEnergyLinearTrans}Past energy function $\mathcal{E}_\gamma^-$ for the model \eqref{eq:quadratic}--\eqref{eq:example1} using a linear state transformation, $\x =\T_1\z$.}
\end{figure}
\begin{figure}
    \centering
    \includegraphics[width=0.475\textwidth]{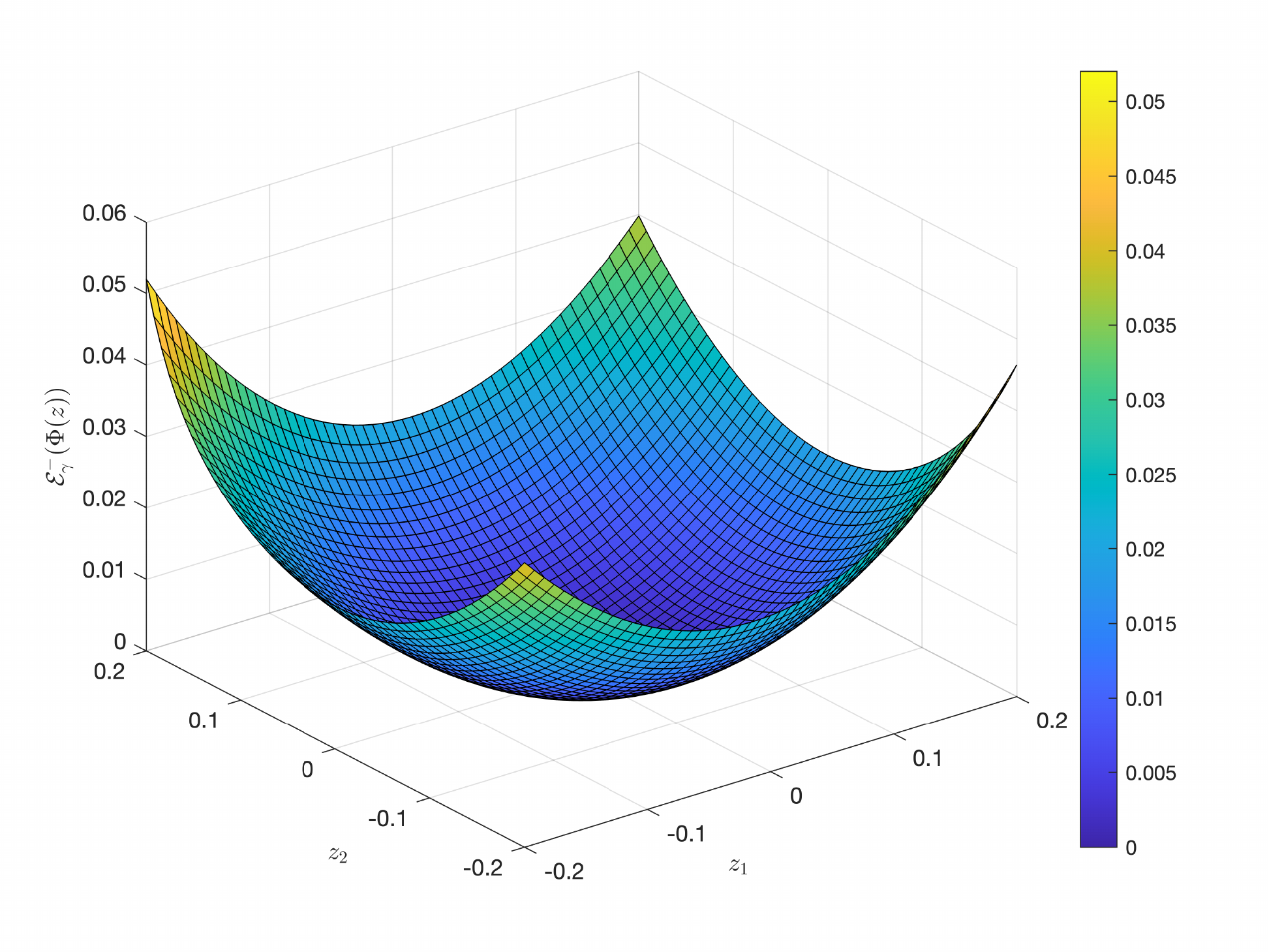}    \caption{\label{fig:pastEnergyQuadraticTrans}Past energy function $\mathcal{E}_\gamma^-$ for the model \eqref{eq:quadratic}--\eqref{eq:example1} using a quadratic state transformation, $\x=\T_1\z + \T_2(\z\otimes\z)$.}
\end{figure}

\begin{table}
  \centering
  \caption{The effect of varying degree transformations on maximum error in the transformed past energy function $\cE_\gamma^-$ over the domain $(-a,a)^2$.}
  \label{tab:TransformedPastEnergy}
\begin{tabular}{c|ccc}
degree & $a=0.2$ & $a=0.05$ & $a=0.01$ \\
\hline
1 & 4.0566e-02 & 1.5498e-04 & 1.0886e-06 \\
2 & 1.8411e-02 & 7.6587e-06 & 9.2309e-09 \\
3 & 1.6217e-02 & 1.7734e-06 & 4.2417e-10 \\
4 & 1.8339e-02 & 4.2031e-07 & 1.9737e-11 \\
5 & 1.6364e-02 & 1.0820e-07 & 1.0743e-12 \\
6 & 1.8547e-02 & 2.3594e-08 & 6.2827e-14 \\
7 & 1.5997e-02 & 1.7912e-08 & 4.7472e-15 \\
8 & 1.9368e-02 & 1.1080e-08 & 1.0545e-15
\end{tabular}
\end{table}

\subsection{Input-output balancing transformation} \label{sec:inputoutput}
The transformation in Section~\ref{sec:inputnormal} brings the system into the input-normal/output-diagonal form, see Theorem~\ref{thm:inputnormal}. The next theorem suggests a transformation that brings the system into the widely-used input-output balanced form, where the singular values appear in both the controllability and observability energy functions. 

\begin{theorem} \cite[Thm. 9]{fujimoto2010balanced} \label{thm:fullybalanced}
	Suppose that the Jacobian linearization of the nonlinear system is controllable, observable, and asymptotically stable. Then there is a neighborhood $\cW$ of the origin and a smooth coordinate transformation $\x = \bar{\Phi}(\bar{\z})$ on $\cW$ converting the controllability and observability energy functions into the form
	\begin{align} \label{eq:energyfct_fullybalanced}
	\cE_c(\bar{\Phi}(\bar{\z}))  & = \frac{1}{2} \sum_{i=1}^n \frac{\bar{z}_i^2}{{\sigma}_i(\bar{z}_i)}, \\
	\cE_o(\bar{\Phi}(\bar{\z}))  & = \frac{1}{2} \sum_{i=1}^n {\sigma}_i(\bar{z}_i) \bar{z}_i^2.
	\end{align}
	Moreover, if $\cW = \Real^n$, then the Hankel norm of the nonlinear system is given by
	\begin{equation}
	    \Vert \Sigma \Vert_H := \sup_{\u \in L_2(\Real^+), \u\neq \bzero} \frac{\Vert \cH(\u)\Vert}{\Vert \u \Vert} = \sup_{\bar{z}_1} {\sigma}_1(\bar{z}_1),
	\end{equation}
    where $\cH$ is the Hankel operator for the nonlinear system. 
\end{theorem}

\vspace{0.2cm}

Such a variable transformation also exits under more technical assumptions for closed-loop system, see \cite[Thm 5.13]{scherpen1996hinfty_balancing}, where the past and future energy functions $\cE_\gamma^-$ and $\cE_\gamma^+$ are transformed into input-output form, and where the underlying system is potentially unstable. 
Note that since each singular value function in \eqref{eq:energyfct_fullybalanced} is associated with a single state component $z_i$, they can be used to decide the truncation of the states, see Section~\ref{sec:ROMs}. 
To obtain the transformation $\bar{\Phi}(\bar{\z})$, first consider  the past energy function in input-normal form, i.e., $\cE_\gamma^-(\Phi(\z)) = \frac{1}{2} \sum_{i=1}^n z_i^2$ from \eqref{eq:balRealCont}. We now need to \emph{scale} the individual states to obtain $\cE_\gamma^-(\bar{\Phi}(\bar{\z}))  = \frac{1}{2} \sum_{i=1}^n \frac{\bar{z}_i^2}{{\sigma}_i(\bar{z}_i)}$. To do this, consider the nonlinear transformation
\begin{equation} \label{eq:fulltrafo}
    \bar{z}_i = z_i \sqrt{\xi_i(z_i)} = \bar{\Phi}_i^{-1}(z_i) \quad \Rightarrow \quad \bar{\Phi}_i(\bar{z}_i) = z_i .
\end{equation}
We insert this transformation into the energy function in \eqref{eq:balRealCont} and obtain
\begin{align}
	\cE_\gamma^-(\bar{\Phi}(\bar{\z})) 
	&  = \frac{1}{2} \sum_{i=1}^n \frac{\bar{z}_i^2}{\xi_i(z_i)} \nonumber \\
	&= \frac{1}{2} \sum_{i=1}^n \frac{\bar{z}_i^2}{\xi_i(\bar{\Phi}_i(\bar{z}_i))} = \frac{1}{2} \sum_{i=1}^n \frac{\bar{z}_i^2}{\sigma_i(\bar{z}_i)},
\end{align}
where the new input-output singular value functions are defined (see \cite[Thm 11]{fujimoto2010balanced}) as 
\begin{equation} \label{eq:fullybalSVs}
    \sigma_i(\bar{z}_i):= \xi_i(\bar{\Phi}_i(\bar{z}_i)) = \xi_i(z_i). 
\end{equation}
This implies that the singular value functions of the input-normal/output-diagonal and the input-output balancing transformations are identical. 
To clarify this point, let us recall the linear case. For an LTI system, let  $\Y_\infty$ and $\X_\infty$ be the solutions to the 
the $\mathcal{H}_\infty$ AREs \eqref{eq:HinftyRiccati1} and \eqref{eq:HinftyRiccati2}, respectively. Then, the input normal form would imply that $\Y_\infty = \I$ and $\X_\infty = \diag(\xi_1(0)^2,\ldots,\xi_n(0)^2)$.
However, since the $\cH_\infty$-characteristic values $\xi_i(0)$, which are the Hankel singular values in the limit $\gamma\rightarrow 1$ in \eqref{eq:HinftyenergyFunctions1}--\eqref{eq:HinftyenergyFunctions2}, are invariant under state-space transformation, in the fully balanced coordinates (after proper scaling) one would have
$\Y_\infty = \X_\infty = \diag(\xi_1(0),\ldots,\xi_n(0))$. 

Next, we consider how the nonlinear state transformation~\eqref{eq:fulltrafo} affects the future energy function of the input-normal/output-diagonal form, $\cE_\gamma^+(\Phi(\z)) = \frac{1}{2} \sum_{i=1}^n z_i^2 \xi_i^2(z_i)$.
Applying the nonlinear transformation from \eqref{eq:fulltrafo}, it follows that the future energy function is automatically transformed into the input-output balanced form as in Theorem~\ref{thm:fullybalanced}:
\begin{align}
	\cE_\gamma^+(\bar{\Phi}(\bar{\z})) & = \frac{1}{2} \sum_{i=1}^n \bar{z}_i^2 \xi_i(z_i) = \frac{1}{2} \sum_{i=1}^n \bar{z}_i^2 \xi_i(\bar{\Phi}(\bar{z}_i)) \nonumber \\
	& = \frac{1}{2} \sum_{i=1}^n \bar{z}_i^2 \sigma_i(\bar{z}_i),
\end{align}
with the the singular value functions from \eqref{eq:fullybalSVs}.

\begin{remark}
The connection to balanced truncation for LTI systems can be further appreciated by writing the energy functions in~\eqref{eq:energyfct_fullybalanced} in the form
\begin{align}
\cE_\gamma^-(\bar{\Phi}(\bar{\z})) & = \frac{1}{2}\bar{\z}^\top \bSigma(\bar{\z})^{-1} \bar{\z}, \\
\quad 
\cE_\gamma^+(\bar{\Phi}(\bar{\z})) & = \frac{1}{2} \bar{\z}^\top \bSigma(\bar{\z}) \bar{\z}
\end{align}
where $\bSigma(\bar{\z})= {\rm{diag}}({\sigma}_1(\bar{z}_1), \ldots, {\sigma}_n(\bar{z}_n))$. The interpretation of this form is to consider the `Gramian' in the balanced coordinates$, \bSigma(\bar{\z})$, as state dependent, as opposed to it being constant in the LTI case. 
\end{remark}

\subsection{Balanced high-dimensional model} \label{sec:balancedFOM}
The nonlinear transformation from Theorem~\ref{thm:fullybalanced} and \eqref{eq:fulltrafo} that brings the dynamical system~\eqref{eq:FOMNL1}--\eqref{eq:FOMNL2} into the input-output balanced coordinate system can be summarized as
\begin{equation} \label{eq:scaledStates}
\begin{cases} 
\x & = \bar{\Phi}(\bar{\z}) = \T_1 \z + \T_{2} \kronF{\z}{2} + \cdots + \T_{k} \kronF{\z}{k} \\
      z_i & = \bar{z}_i /\sqrt{\sigma_i(\bar{z}_i)} .
\end{cases}
\end{equation}
Note that the transformation matrices $\T_i$ did not change, as the input-normal/output-diagonal form is already diagonalized, but the scaling has. In other words, using \eqref{eq:fullybalSVs} each new state is now scaled by $\sqrt{\xi_i(z_i)}$ to get the fully balanced form from Theorem~\ref{thm:fullybalanced}. 
In the LTI case, $\x= \T_1\z$ from \eqref{eq:trafo} yields the input-normal/output-diagonal and by an additional scaling of the state, $\bar{z}_i/\sqrt{\xi_i(0)}=z_i$ we obtain the transformation $\x= \T_1{\z}= \T_1\bSigma^{-1/2}\bar{\z}$ which is the input-output balanced form in \eqref{eq:scaledStates}.

The dynamical system when transformed with the input-output balancing transformation $\x = \bar{\Phi}(\bar{\z})$ is 
\begin{equation} \label{eq:balFOM}
      \bar{\J}(\bar{\z}) \dot{\bar{\z}} = \f (\bar{\Phi}(\bar{\z})) + \g(\bar{\Phi}(\bar{\z})) \u,
\end{equation}
where the Jacobian $\bar{\J}(\bar{\z})\in \Real^{n\times n}$ of the state-space transformation is given by
\begin{align} \label{eq:Jacobian}
    \bar{\J}(\bar{\z})  
     := & \frac{\text{d}{\bar{\Phi}(\bar{\z})}}{\text{d}\bar{\z}} \\
    = & \T_1 + 2\T_2(\bar{\z}\otimes \I) + 3\T_3 (\bar{\z}\otimes \bar{\z}\otimes \I) + \ldots, \nonumber
\end{align}
where we used the fact that since we compute $\T_k$ with symmetric coefficients, it follows that, e.g., $\T_2(\bar{\z}\otimes\I) = \T_2(\I\otimes\bar{\z})$. See Definition~\ref{def:sym} for more details on symmetric coefficients.
The Jacobian can be computed explicitly without numerical approximation, as it is the derivative of a polynomial transformation. Remarkably, the Jacobian $\frac{\text{d}{\Phi(\z})}{\text{d}\z}$ of the input-normal/output-diagonal transformation \eqref{eq:trafo} and the Jacobian $\frac{\text{d}{\bar{\Phi}(\bar{\z})}}{\text{d}\bar{\z}}$ of the input-output balancing transformation \eqref{eq:scaledStates} have the same coefficient matrices $\T_i$, which significantly simplifies the ROM simulation.
In the next section we revisit the standard balancing transformation that typically results in conditioning problems.  We then introduce a novel approximation to simultaneously balance-and-reduce nonlinear systems in a well-conditioned and computationally efficient way.

\section{Balanced truncation model reduction via nonlinear transformations} \label{sec:ROMs}
In this section, we reduce the dimensionality of the fully balanced model
\eqref{eq:balFOM}.  To determine the reduced dimension $r$ of the ROM, we look for a significant gap in the $\cH_\infty$ singular value functions, i.e., we look for the reduced dimension $r$ such that
\begin{equation}
    \max_{\bar{z}_r} \sigma_r(\bar{z}_r) \gg \max_{\bar{z}_{r+1}} \sigma_{r+1}(\bar{z}_{r+1})
\end{equation}
(at a minimum we require that `$>$' holds) in a neighborhood of the origin. This indicates that the state components $\bar{z}_1, \bar{z}_2, \ldots, \bar{z}_r$ are more important in terms of the past and future energy functions $\cE_\gamma^-$ and $\cE_\gamma^+$ than the states $\bar{z}_{r+1}, \bar{z}_{r+2}, \ldots, \bar{z}_n$.  We therefore set $\bar{z}_{r+1} = \bar{z}_{r+2} = \ldots = \bar{z}_n = 0$ in the balanced coordinates and define the reduced state vector as 
\begin{equation} \label{eq:reducedStates}
    \bar{\z}_r = \bPsi_r^\top  \ \bar{\z} =  [\bar{z}_1, \bar{z}_2, \ldots \bar{z}_r]^\top, \quad \bPsi_r = \begin{bsmallmatrix} \I_r & \bzero \end{bsmallmatrix}^\top \in \Real^{n\times r}.
\end{equation}
The next Section~\ref{sec:balthenreduce} presents the originally-proposed balanced ROM from \cite{scherpen1993balancing,scherpen1996hinfty_balancing}. Section~\ref{sec:simultaneiousBalancingROM} proposes a novel and numerically more efficient and better conditioned approximate strategy to compute the nonlinear ROMs corresponding to this truncation strategy. Section~\ref{sec:NLmanifolds} suggests a different perspective of the nonlinear ROM, namely the approximation on a nonlinear balanced manifold.

\subsection{Balance-then-reduce approach} \label{sec:balthenreduce}
The balance-\textit{then}-reduce strategy suggested in~\cite{scherpen1993balancing,scherpen1996hinfty_balancing} first computes the full balancing transformation, and then truncates the resulting fully balanced system. Applying this to equation~\eqref{eq:balFOM} yields
\begin{align} 
      \dot{\bar{\z}}_r  = & \underbrace{\bPsi_r^\top [\bar{\J}([\bar{\z}_r,\bzero])]^{-1}  \f (\bar{\Phi}([\bar{\z}_r,\bzero]))}_{=:\f_r(\bar{\z}_r)} \nonumber \\
      & +  \underbrace{\bPsi_r^\top [\bar{\J}([\bar{\z}_r,\bzero])]^{-1} \g(\bar{\Phi}([\bar{\z}_r,\bzero]))}_{=:\g_r(\bar{\z}_r)} \u, \label{eq:balROM1} \\
      \y_r  = & \underbrace{\h ( \bar{\Phi}([\bar{\z}_r, \bzero]))}_{=:\h_r(\bar{\z}_r)}. \label{eq:balROM2}
\end{align}
The high-dimensional state is reconstructed as $\x \approx \bar{\Phi}([\bar{\z}_r, \bzero])$. 
A goal of balanced truncation  is to obtain ROMs that are balanced in the reduced coordinates and that retain properties of the FOM, such as stability. The following theorems show that obtaining such results depends on which energy functions are used for balancing. First, we consider the case of balancing the open-loop energy functions from~\eqref{eq:limitto1}.
\begin{theorem} \cite[Thm. 10]{fujimoto2010balanced}  \label{thm:prop1}
Consider the nonlinear dynamical system~\eqref{eq:FOMNL1}--\eqref{eq:FOMNL2}. 
Suppose that 
\begin{enumerate}
\item the open-loop controllability and observability energy functions $\cE_c$ and $\cE_o$ exist,
\item the matrices $\left [\frac{\partial^2 \cE_c(\x)}{\partial x_i \partial x_j}(\bzero)\right ]_{i,j =1, \ldots n}$ and $\left [\frac{\partial^2 \cE_o(\x)}{\partial x_i \partial x_j}(\bzero)\right ]_{i,j =1, \ldots n}$ are positive definite, 
\item the eigenvalues of $\left [\frac{\partial^2 \cE_c(\x)}{\partial x_i \partial x_j}(\bzero)\right ]^{-1} \left[\frac{\partial^2 \cE_o(\x)}{\partial x_i \partial x_j}(\bzero)\right ]$ are distinct. 
\end{enumerate}
Then, with the input-output balancing transformation $\bar{\Phi}(\bar{\z})$, the controllability function $\cE_{c,r}$ and observability energy function $\cE_{o,r}$ of the balanced ROM~\eqref{eq:balROM1}--\eqref{eq:balROM2} satisfy
\begin{align}
    \cE_{c,r}(\bar{\z}_r) = \cE_c([\bar{\z}_r,\bzero]), \quad \cE_{o,r}(\bar{\z}_r) = \cE_o([\bar{\z}_r,\bzero]).
\end{align}
Moreover, the singular value functions $\sigma_{i,r}(\bar{z}_i)$ of the ROM similarly satisfy
\begin{align}
    \sigma_{i,r}(\bar{z}_i) = \sigma_i(\bar{z}_i),
\end{align}
and the energy functions of the ROM are balanced in the sense of Theorem~\ref{thm:fullybalanced}.
\end{theorem}

\begin{remark}
    From Theorem~\ref{thm:prop1} we see that under suitable assumptions, the open-loop energy functions are preserved at the ROM level. This implies that the ROM inherits the local asymptotic stability of the FOM~\cite[Thm 5.3]{scherpen1993balancing}. In some special cases global asymptotic stability can be guaranteed, see~\cite[Thm 5.4]{scherpen1993balancing}.
pp\end{remark}

The next theorem addresses the case when balancing is performed with the closed-loop $\cH_\infty$ (past and future) energy functions~\eqref{eq:HinftyenergyFunctions1}--\eqref{eq:HinftyenergyFunctions3}. In this scenario, an extra condition is needed so that the ROMs remain balanced in the reduced coordinates as well. 
\begin{theorem} \cite[Thm 6.1.]{scherpen1996hinfty_balancing}
    Let $\check{\z} = [\bar{z}_{r+1}, \ldots, \bar{z}_{n}]$ be the vector of state components that are truncated from the FOM. Correspondingly, let the input-output balanced system \eqref{eq:balFOM} be partitioned as $[\dot{ \bar{\z}}_r; \dot{\check{\z}}]^\top = [\f_1(\bar\z); \f_2(\bar\z)]^\top + [\g_1(\bar\z); \g_2(\bar\z)]^\top \u$. 
    Suppose that the closed-loop energy functions $\cE_\gamma^-$ and $\cE_\gamma^+$ from ~\eqref{eq:HinftyenergyFunctions1}--\eqref{eq:HinftyenergyFunctions3} exist.
    With the input-output balancing transformation $\bar{\Phi}(\bar{\z})$, the past energy function of the balanced ROM~\eqref{eq:balROM1}--\eqref{eq:balROM2} satisfies
    \begin{align}
        \cE_\gamma^-(\bar{\z}_r) = \cE_\gamma^-([\bar{\z}_r,\bzero]).
    \end{align}
    Assume further that
    \begin{align}
    \frac{\partial \cE_\gamma^+}{\partial \check{\z}}(\bar{\z}_r,\bzero)\f_2(\bar{\z}_r,\bzero)=0, \quad
    \frac{\partial \cE_\gamma^+}{\partial \check{\z}}(\bar{\z}_r,\bzero)\g_2(\bar{\z}_r,\bzero)=0
    \end{align}
    holds; then the future energy function satisfies
    \begin{align}
        \qquad \cE_\gamma^+(\bar{\z}_r) = \cE_\gamma^+([\bar{\z}_r,\bzero]).
    \end{align}
\end{theorem}

Since the original system  is not assumed stable in the $\cH_\infty$ framework, an interesting property to study is whether the reduce-\textit{then}-design strategy for the $\cH_\infty$ suboptimal control is guaranteed to produce a suboptimal control for the $\cH_\infty$-balanced ROM. For conditions when that holds true, we refer to \cite[Thm 6.3]{scherpen1996hinfty_balancing}.

The evaluation of the polynomials terms on the right-hand side of the ROM~\eqref{eq:balROM1} can be simplified for the special case $\f(\z) = \A\z + \N(\z\otimes\z)$, which yields
\begin{align}
    & \f (\bar{\Phi}([\bar{\z}_r,\bzero]))= \nonumber \sum_{i=1}^k [ \A \T_i \kronF{\I_{n\times r}}{i}] \kronF{\bar{\z}_r}{i} \nonumber \\
    & \qquad + \N \left (\sum_{i=1}^k [\T_i \kronF{\I_{n\times r}}{i}] \kronF{\bar{\z}_r}{i} \otimes \sum_{i=1}^k [\T_i \kronF{\I_{n\times r}}{i}] \kronF{\bar{\z}_r}{i}  \right ).
\end{align}
Nevertheless, simulating the ROM~\eqref{eq:balROM1} is computationally expensive and will likely result in inverting an ill-conditioned matrix, which is in analogy to the linear case. In the next section, we propose a 
better conditioned and computationally more efficient implementation of the nonlinear ROM.

\subsection{Simultaneous balancing and reduction} \label{sec:simultaneiousBalancingROM}
The computation of $\bPsi_r^\top [\bar{\J}([\bar{\z}_r,\bzero])]^{-1}$ in the balanced ROM \eqref{eq:balROM1}--\eqref{eq:balROM2} requires inverting the full Jacobian followed by truncation.  This has a computational and a numerical disadvantage. 
First, this strategy requires a high number of floating point operations to form the full $\T_j \in \Real^{n\times n^j}, \ j=2, \ldots, k$ as in Theorem~\ref{thm:polyCoeffsTrafo}.
Second, computing $\T_1$ requires inversions that are often ill-conditioned for large-scale systems (the main focus of this work) since it requires inverting all the $\cH_\infty$-characteristic values, including the smallest ones. In the linear case, the balance-\textit{then}-reduce strategy is well known to be ill-conditioned due to small Hankel singular values, see, e.g.,~\cite[Sec. 7.3]{antoulas05} and a remedy is to perform simultaneous model reduction and truncation.

We suggest a new computational framework for the nonlinear case following these ideas from the linear case.  Our goal is to 
compute the truncated versions of the linear transformations and higher-degree tensors $\{\T_i\}_{i=1}^k$  from Theorem~\ref{thm:polyCoeffsTrafo} directly without computing the full-order quantities.
We begin with deriving the form of the nonlinear ROM for the FOM~\eqref{eq:FOMNL1}--\eqref{eq:FOMNL2} when a general polynomial state transformation and simultaneous reduction is applied. In other words, we are approximating the FOM~\eqref{eq:FOMNL1}--\eqref{eq:FOMNL2} on a nonlinear (here: balanced) manifold such that $\x \approx \Phi_r(\bar{\z}_r)$, where the mapping $\Phi_r$ defines a manifold. The following result applies to \textit{any} polynomial state transformation.
\begin{proposition} \label{prop:redman}
Consider the nonlinear system~\eqref{eq:FOMNL1}--\eqref{eq:FOMNL2}. Let $\T_{k,r} \in \Real^{n\times r^k}$ be (truncated) transformation matrices with symmetric coefficients and $\bar{\z}_r\in \Real^r$ denote the reduced state. Define the embedding $\Phi_r: \Real^r \mapsto \Real^n$ via
\begin{equation}
    \x \approx \Phi_r(\bar{\z}_r) := \T_{1,r} \bar{\z}_r + \T_{2,r} \kronF{\bar{\z}_r}{2} + \cdots + \T_{k,r} \kronF{\bar{\z}_r}{k}, \label{eq:NlStateReduction}
\end{equation}
for which the reduced Jacobian can be computed explicitly via
\begin{align} \label{eq:newJr}
 \J_r(\bar{\z}_r)  :=&  \frac{{\rm{d}}\Phi_r(\bar{\z}_r)}{{\rm{d}}\bar{\z}_r} \nonumber \\
    =&  \T_{1,r} + 2\T_{2,r}(\bar{\z}_r\otimes \I) \nonumber \\
    & + 3 \T_{3,r} (\bar{\z}_r\otimes \bar{\z}_r\otimes \I) + \cdots \in \Real^{n\times r}.
\end{align}
Thus, the nonlinear ROM for $\bar{\z}_r\in\Real^r$ is
\begin{align} 
      \dot{\bar{\z}}_r 
      & = \underbrace{\J_r(\bar{\z}_r)^{\dagger} \f(\Phi_r(\bar{\z}_r))}_{=:\f_r(\bar{\z}_r)} + \underbrace{ \J_r(\bar{\z}_r)^{\dagger} \g(\Phi_r(\bar{\z}_r))}_{=:\g_r(\bar{\z}_r)} \u \label{eq:NLROM1} \\
      \y_r & = \underbrace{\h(\Phi_r(\bar{\z}_r))}_{=:\h_r(\bar{\z}_r)}. \label{eq:NLROM2}
\end{align}
\end{proposition}
\vspace{0.1cm}
where $[\cdot]^\dagger$ denotes the Moore-Penrose pseudoinverse.
\begin{proof}
Given $\Phi_r(\bar{\z}_r)$ in~\eqref{eq:NlStateReduction} and  using the symmetry arguments as in \eqref{eq:Jacobian} (which result from our symmetrized computation of the $\T_{i,r}$), we can directly verify the polynomial form of the Jacobian following \eqref{eq:Jacobian} with $\bar{\z}_r$ instead of $\bar{\z}$, thus obtaining~\eqref{eq:newJr}. 
With this reduced Jacobian defined, the model~\eqref{eq:balFOM} becomes
\begin{align} 
      \J_r(\bar{\z}_r) \dot{\z}_r  = \f(\Phi_r(\bar{\z}_r))+ \g(\Phi_r(\bar{\z}_r)) \u \in \Real^n.
\end{align}
We left-multiply the last equation with the pseudo-inverse of the Jacobian to get the ROM~\eqref{eq:NLROM1}--\eqref{eq:NLROM2}.
\end{proof}

The version of the ROM formulated in
Proposition~\ref{prop:redman} resolves the ill-conditioning issue since it no longer requires inverting the full Jacobian. Now, what is needed is a strategy that could compute the coefficient matrices, $\T_{i,r} \in \Real^{n\times r^i}$ of the nonlinear embedding $\Phi_r$ without needing to construct full coefficient matrices, $\T_{i}\in \Real^{n\times n^i}$. Following Theorem~\ref{thm:polyCoeffsTrafo} and Algorithm~\ref{algo:inputnormalTrafo} and motivated by the linear balanced truncation framework, in the next result we propose a new strategy to compute the matrices $\T_{i,r}$ of the nonlinear balancing transformation~\eqref{eq:NlStateReduction}.
Before stating this result, we point out that the ROM computation in~\eqref{eq:NLROM1} still requires evaluating the full $\f(\cdot)$ and $\g(\cdot)$ as nonlinear functions acting on vectors of dimension $n$. We revisit and resolve this issue---known as the lifting bottleneck---in Remark~\ref{rem:deim}.
\begin{proposition} \label{prop:TruncatedBT}
(\textit{Truncated (approximate) balanced transformation})
Let $\v_i,\w_i$ be the vectors of polynomial coefficients for the energy functions from \eqref{eq:pastenerexp}, \eqref{eq:wi_coeffs}. Let $\R, \L$ be their Cholesky factors, i.e., $\V_2 = \R \R^\top$ and $\W_2 = \L \L^\top$. Let $\L^\top \R^{-\top} = \mathcal{U} \bXi \mathcal{V}^\top$  be the singular value decomposition and define 
    $$
    \mathcal{U}_r = \mathcal{U}(:,1:r),~~\bXi_r = \bXi(1:r,1:r),~~\mathcal{V}_r = \mathcal{V}(:,1:r).
    $$ 
   {Then, the coefficient matrices of the nonlinear embedding $\Phi_r: \Real^r \mapsto \Real^n$  in~\eqref{eq:NlStateReduction} are} 
\begin{align}
    \T_{1,r} &= \R^{-\top} \mathcal{V}_r \in \Real^{n\times r},\\
    \T_{1,r}^{\dagger} & = \bXi_r^{-1} \mathcal{U}_r^\top \L^\top \in \Real^{r\times n}, \ \text{(left inverse)}\\
    \T_{2,r} & = -\frac{1}{2} \T_{1,r} \ {\normalfont \text{unvec}} \left (\kronF{[\T_{1,r}}{3}]^\top \v_3 \right )^\top \in \Real^{n\times r^2},
\end{align}
and more generally, for $2\leq k$,
{\small
\begin{subequations}
\begin{align}
    \T_{k,r} &= -\frac{1}{2} \T_{1,r} {\normalfont \text{unvec}} (\M_{k,r} )^\top \in \Real^{n\times r^k},~\mbox{where} \\
   \M_{k,r} & =  \sum_{\substack{i,j>1\\ i+j=k+1}}  {\normalfont \text{vec}} \left ( \T_{j,r}^\top \V_2 \T_{i,r} \right ) + \sum_{i=3}^{k+1}  \cT_{i,k+1}^\top \v_i.
\end{align}
\end{subequations}
}
\end{proposition}
We highlight that due to the way we compute $\T_{i,r}$ above, the embedding, and hence Jacobian, is better conditioned and faster to evaluate than in the balance-\textit{then}-reduce strategy in Section~\ref{sec:balthenreduce}. The next theorem shows that the linear part of the transformation 
diagonalizes the product of the Gramians, as in the linear case. 
\begin{proposition}
{The approximate balancing transformation}~\eqref{eq:NlStateReduction} with $\T_{i,r}$ from Proposition~\ref{prop:TruncatedBT} satisfies
    \begin{equation}
        \bXi_r^2 = {\normalfont \text{diag}} (\xi_1^2(0), \ldots, \xi_r^2(0)) =  \T_{1,r}^{\dagger} \V_2^{-1} \W_2 \T_{1,r}.
    \end{equation}
\end{proposition}
\begin{proof}
Given the definitions of Proposition~\ref{prop:TruncatedBT} we directly compute
\begin{align*}
& \T_{1,r}^{\dagger} \V_2^{-1} \W_2 \T_{1,r} \\
& = \bXi_r^{-1} \mathcal{U}_r^\top \L^\top \R^{-\top} \R^{-1} \L \L^\top \R^{-\top} \mathcal{V}_r \\
& = \bXi_r^{-1} \mathcal{U}_r^\top \mathcal{U}_r \bXi_r \mathcal{V}_r^\top \mathcal{V}_r \bXi_r \mathcal{U}_r^\top \mathcal{U}_r \bXi_r \mathcal{V}_r^\top \mathcal{V}_r \\
& = \bXi_r^{-1} \bXi_r^3 = \bXi_r^2
\end{align*}
\end{proof}

With a numerically well-conditioned ROM~\eqref{eq:NLROM1}--\eqref{eq:NLROM2} and a computationally efficient way to compute the nonlinear balanced manifold, see Proposition~\ref{prop:TruncatedBT}, we summarize the steps to arrive at a simultaneously balanced-\underline{and}-reduced ROM in Algorithm~\ref{alg:fullBTROM}.

\begin{algorithm}[!htb]
	\caption{Computation of nonlinear input-output $\cH_\infty$-balanced ROM.}
	\label{alg:fullBTROM}
	\begin{algorithmic}[1]
		\Require Constant $\gamma>\gamma_0\geq0$, $\gamma \neq 1$; polynomial degrees $d> k>\ell$; reduced model order $r$
		\Ensure Input-output nonlinear $\cH_\infty$-balanced ROM~\eqref{eq:NLROM1}--\eqref{eq:NLROM2}.
	   \State Obtain a polynomial representation (or approximation) of the past and future energy functions $\cE_\gamma^-(\x)$ and $\cE_\gamma^+(\x)$, i.e.,  coefficients $\left \{ \v_i \right \}_{i=2}^d$ and $\left \{ \w_i \right \}_{i=2}^d$ in \eqref{eq:pastenerexp}, \eqref{eq:wi_coeffs} (e.g., through Part~1 of this paper\cite{KGB_NonlinearBT_Part1})
        \State Compute the truncated polynomial coefficient matrices $\left \{ \T_{i,r}\right \}_{i=1}^k$ for $\x \approx \Phi_r(\bar{\z}_r)$ from \eqref{eq:NlStateReduction} following Proposition~\ref{prop:TruncatedBT}. 
        \State Symmetrize the coefficients $\left \{ \T_{i,r}\right \}_{i=1}^r$ (see Remark~\ref{remark:kroneckerPoly})
        \State Assemble the nonlinear ROM functions $\f_r(\bar{\z}_r), \g_r(\bar{\z}_r), \h_r(\bar{\z}_r)$ as in ~\eqref{eq:NLROM1}--\eqref{eq:NLROM2} with the explicit Jacobian in \eqref{eq:newJr}.
	\end{algorithmic}
\end{algorithm}

Algorithm~\ref{alg:fullBTROM} requires a choice of ROM dimension $r$. This can be done by plotting all state-dependent singular value functions from \eqref{eq:expandSval}, whose coefficients are obtained via Algorithm~\ref{algo:inputnormalSvals}, and deciding on an ordering of the functions in magnitude. However, this procedure increases the (offline) cost of the algorithm, as all $\c_i$ and $\T_i$, for $i=1,2,\ldots,n$ are required. To reduce the offline cost, it may be enough in some cases to only consider the constant terms $\xi_i(0)$ in the singular value functions to make a truncation decision based on their decay, as is done in balanced truncation for LTI systems.

\begin{remark} \label{rem:deim}
Algorithm~\ref{alg:fullBTROM} helps to resolve the computational complexity and numerical ill-conditioning issues of the balance-then-reduce strategy using the simultaneous (approximate) balancing-\underline{and}-reduce strategy. 
However, as stated earlier, we point out that the ROM computation in~\eqref{eq:NLROM1}-\eqref{eq:NLROM2} still requires evaluating the full $\f(\cdot)$ and $\g(\cdot)$ as nonlinear functions acting on vectors of dimension $n$.
This is a common issue in nonlinear model reduction not specific to nonlinear balanced truncation.
If the inputs and outputs are linear, then 
$\g_r(\bar{\z}_r) = \J_r(\bar{\z}_r)^{\dagger}\B$ and 
$\h_r(\bar{\z}_r)=\C \Phi_r(\bar{\z}_r)$ where all products of $\C$ and $\T_{k,r}$ can be precomputed.
However, still the ROM in~\eqref{eq:NLROM1}  requires evaluating the full nonlinearity. 
This could be circumvented by a hyper-reduction procedure, e.g., DEIM~\cite{chaturantabut2010nonlinear} and its variants~\cite{DrmG16,DrmS18}.
Only when the, inputs, outputs \textit{and} dynamics are linear, then $\Phi(\bar{\z}_r) = \T_{1,r}\bar{\z}_r$ and $\J_r(\bar{\z}_r) = \T_{1,r}^\dagger$, so the linear ROM becomes
\begin{align}
      \dot{\bar{\z}}_r  = \T_{1,r}^\dagger \A \T_{1,r}\bar{\z}_r +  \T_{1,r}^\dagger \B \u, \quad \y_r = \C \T_{1,r} \bar{\z}_r,
\end{align}
where everything can be precomputed and the ROM can be simulated without any reference to the full dimension~$n$.
\end{remark}

\subsection{Nonlinear balanced manifold ROMs} \label{sec:NLmanifolds}
The described nonlinear balanced truncation approach performs model reduction on an $r$-dimensional polynomially-nonlinear manifold with specific structure. 
Nonlinear balanced truncation first performs a nonlinear balancing transformation $\x = \bar{\Phi}(\z)$, which is followed by truncation.
Thus, $\x \approx \Phi_r(\bar{\z}_r) := \T_{1,r} \bar{\z}_r + \T_{2,r} \kronF{\bar{\z}_r}{2} + \cdots + \T_{k,r} \kronF{\bar{\z}_r}{k}$ as in \eqref{eq:NlStateReduction}, which defines the nonlinear balancing manifold
\begin{equation}
    \mathcal{M}:= \left \{ \Phi_r(\bar{\z}_r)\ : \ \bar{\z}_r \in \Real^r \right \}\subseteq \Real^n.
\end{equation}
Figure~\ref{fig:balancedManifold} illustrates that the system evolves on this balanced nonlinear manifold (the $\bar{z}_i$-coordinates) of the original state space (the $x_i$-coordinates) such that the past and future energy (or controllability/observability energy) of a reduced state can be assessed via the well-known balanced form as outlined in Theorem~\ref{thm:fullybalanced}. This fact then allows for truncation of individual states, as each state has associated energies. In the $\cH_\infty$-balancing case, the retained ROM states are easy to filter and easy control. In the open-loop case, the retained states are easy to reach and easy to observe.
\begin{figure}
    \centering
    \includegraphics[width=0.5\textwidth,trim={0 0.6cm 0.15cm 0},clip]{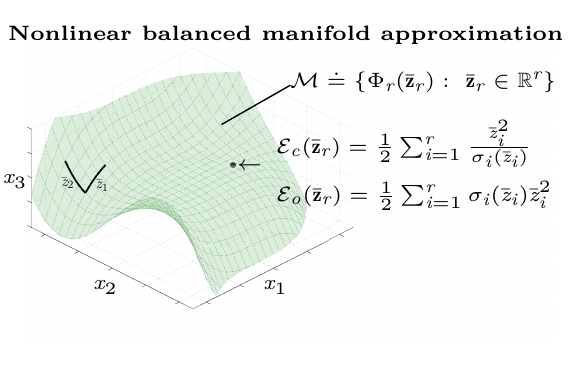}
    \caption{The input-output balanced ROM evolves on a nonlinear manifold, $\mathcal{M}$, such that the controllability/observability (or past/future) energy of a state can be measured in the well-known form that involves the singular value functions $\sigma_i(\bar{z}_i)$.}
    \label{fig:balancedManifold}
\end{figure}

Nonlinear manifold ROMs have become popular recently, as they address the challenges that linear subspace ROMs can face due to slow decay of the Kolmogorov $n$-width, see the survey article~\cite{peherstorfer22breakingKolmogorov}.
For instance, the authors in \cite{carlberg20autoencoderROM,fresca2021comprehensive} use a fully nonlinear autoencoder approach to define the nonlinear manifold ROM $\x \approx \psi(\bar{\z}_r)$, while the  work in~\cite{jain2017quadratic,geelen2022nonlinearManifoldOPINF,barnett2022quadraticManifold} uses quadratic manifolds of the kind $\x \approx \V \bar{\z}_r + \bar{\V}(\bar{\z}_r\otimes \bar{\z}_r)$. While these nonlinear manifold approaches improve the state-approximation error, they do not take into account the system-theoretic (observability/controllability) energies in the reduction, in turn neglecting the full dynamics of the input-to-output map.

\section{Numerical Results: Burgers Equation} \label{sec:numerics}
We present a proof-of-concept example that tests the effectiveness of the algorithms presented in Sections~\ref{sec:BalModels} and \ref{sec:ROMs}.  The finite-dimensional model is generated from finite element discretizations of a PDE.  The PDE of Burgers type is chosen for its quadratic nonlinear terms. This test problem has a long history in the study of control for 
distributed parameter systems, e.g.~\cite{thevenet2009nonlinear}, including the 
development of effective computational methods, e.g.~\cite{burns1990control}, and reduced-order models, e.g.~\cite{kunisch1999cbe}. 
Consider the controlled Burgers equation
\begin{align}
\label{eq:burgers}
  \omega_t(\xi,t) &= \epsilon \omega_{\xi\xi}(\xi,t) - \frac{1}{2}\left(\omega^2(\xi,t)\right)_\xi +  
  \sum_{j=1}^m b_j^m(\xi) u_j(t),\\
  \label{eq:averagingObservation}
  y_i(t) &= \int_{\chi_{[(i-1)/p,i/p]}}\hspace{-2em}\omega(\xi,t) \text{d} \xi, \qquad i=1,\ldots,p, 
\end{align}
with zero Dirichlet boundary conditions. Control inputs are described using the characteristic function $\chi$ as
$b_j^m(\xi) = \chi_{[(j-1)/m,j/m]}(\xi)$ and the outputs are spatial averages of the solution over $p$ equally-spaced subdomains.  We  discretize the state equation with $n+1$ linear finite elements leading to an $n$-dimensional state vector.  The discretized system has the form
\begin{align}\label{eq:discreteCI}
  \widetilde{\bf E}\dot{\boldsymbol \omega} &= \widetilde{\bf A}{\boldsymbol \omega} + \widetilde{\bf N}_2\left( {\boldsymbol \omega}\otimes{\boldsymbol \omega} \right) + \widetilde{\bf B}{\bf u}\\
  {\bf y} &= \widetilde{\bf C}{\boldsymbol \omega},
\end{align}
where ${\boldsymbol \omega}(t)$ are the coefficients of the finite element approximation to $\omega(x,t)$.  To place this in the form~\eqref{eq:FOMNL1}--\eqref{eq:FOMNL2} where the mass matrix is the identity, we introduce
the change of variables ${\bf x} = \widetilde{\bf E}^{1/2}{\boldsymbol \omega}$ where $\widetilde{\bf E}^{1/2}$ is a matrix square root of the finite element mass matrix.  Then defining
$\A=\widetilde{\bf E}^{-1/2}\widetilde{\A}\widetilde{\bf E}^{-1/2}$, $\B=\widetilde{\bf E}^{-1/2}\widetilde{\bf B}$, ${\bf C} = \widetilde{\bf C}\widetilde{\bf E}^{-1/2}$, ${\bf N}_2 = \widetilde{\bf N}_2(\widetilde{\bf E}^{-1/2}\otimes\widetilde{\bf E}^{-1/2})$ leads to a system in the required form.

\subsection{Singular value functions}
In the first experiment, we use the values $\epsilon=0.05$, $n=16$, $m=4$, $p=4$, and $\gamma=3$, and we compute the quartic approximations to the past and future energy functions as described in Part~1 of this paper, \cite{KGB_NonlinearBT_Part1}, to find the coefficients in \eqref{eq:pastenerexp} and \eqref{eq:wi_coeffs}.  Using these coefficients, we compute the cubic transformation tensors $\T_{1}, \T_{2}, \T_{3}$ using Algorithm~\ref{algo:inputnormalTrafo}.  Finally, we compute quadratic approximations to the singular value functions as in Algorithm~\ref{algo:inputnormalSvals}.  The first eight of these are plotted in Figure~\ref{fig:svf}.  
For this problem, we observe that the first four singular value functions are significantly larger in magnitude than the remainder.
Moreover, the ranks of the singular value functions change across the variable range. We observe interesting behavior in the fifth singular value function, with a quadratic approximation that rapidly goes negative in a small region around zero.  Although not plotted here for space, the cases with $p=1$ and $p=2$ outputs are different.  
In the $p=1$ case, the singular value functions do not cross in the same parameter interval; there is an order of magnitude separation of $\xi_1$, $\xi_2$, and $\xi_3$; there are nearly two orders of magnitude between $\xi_5$ and $\xi_6$.  For the case $p=2$, there is a strong separation of the singular value functions near 0, but $\xi_1$ and $\xi_2$ do cross at about $|z|\approx 0.1$. We investigate the quality of ROMs with growing dimension in the next section and correlate the results with the qualitative behavior of the singular value functions. 
\begin{figure}
    \centering
    \includegraphics[width=0.475\textwidth]{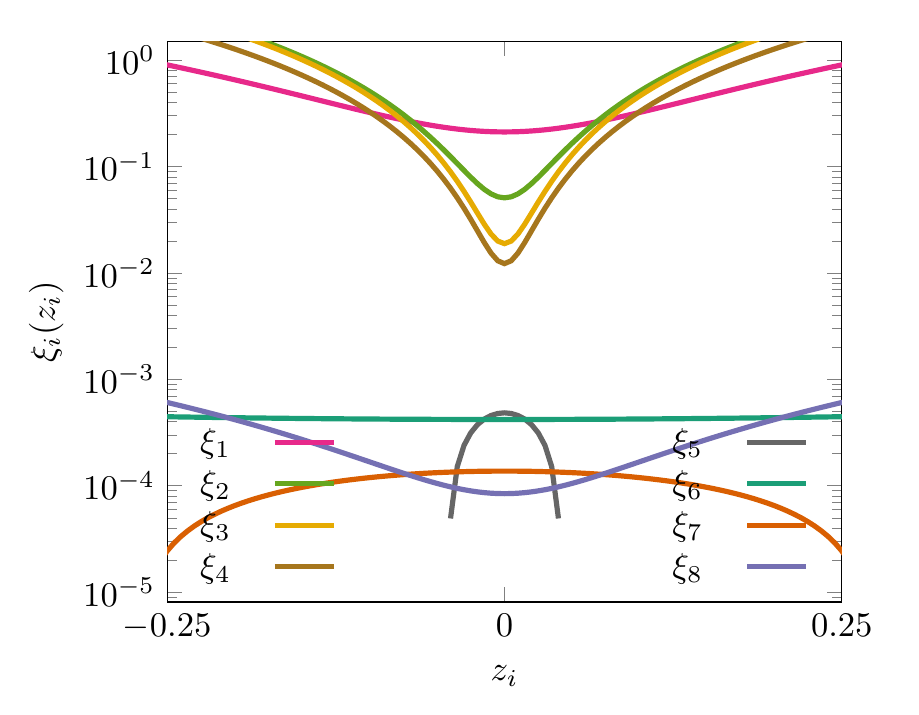}    \caption{\label{fig:svf}Quadratic approximations to the first eight singular value functions for the Burgers example.}
\end{figure}
%

\subsection{Output behavior}
We  simulate the outputs for the system described above. We present results for different ROM-dimensions $r$ and different degree $k$ nonlinear transformations.  We simulate the ROMs from $t=0$ to $t=10$ and report the relative error for each output, $e_i$, defined using
\begin{displaymath}
  e_i = \frac{\left(\int_0^{10} \left| y_i(t)-y_{i,r}(t) \right|^2\ {\rm d}t\right)^{1/2}}{\left(\int_0^{10} \left| y_i(t) \right|^2\ {\rm d}t\right)^{1/2}} , \ i=1,\ldots,p, 
\end{displaymath}
where $y_{i,r}$ is the $i$th output of the ROM of order $r$. For all of our tests, $m>1$ and we use the inputs
\begin{displaymath}
  u_i(t) = \left\{ \begin{array}{cl}
  0.002\tan^{-1}(t) + 0.001\sin(t), & i=1 \\
  0, & \mbox{otherwise}
  \end{array} \right. .
\end{displaymath}

\begin{table}
\caption{\label{tab:relErr}Relative errors ($e_1$) for the $m=4$ and $p=1$ case.}
\centering
\begin{tabular}{cccc}
\hline
$r$ & $k=1$ & $k=3$  & $k=5$ \\
\hline
1  & 0.0714831 & 0.0714814 & 0.0713882 \\
2  & 0.0036861 & 0.0036778 & 0.0031076 \\
3  & 0.0026888 & 0.0026784 & 0.0026665 \\
4  & 0.0024333 & 0.0024288 & 0.0024238 \\
5  & 0.0024095 & 0.0024032 & 0.0023853 \\
\hline
\end{tabular}
\end{table}

In Table~\ref{tab:relErr}, we present the performance of the ROMs for increasing model order and degree of the nonlinear transformation.  As expected, we see the general trend of a decreasing error with increasing the reduced dimension and the degree of the transformation. We do not include in Table~\ref{tab:relErr} the results past $r=5$ as the relative error does not improve further using any degree of the transformations as expected from the singular value function behavior in Figure~\ref{fig:svf}.

In the next experiment, we consider a problem where the linear portion of the model is more significant $\epsilon=0.1$ and we take two model outputs $p=2$ (with the same number of inputs $m=4$).  The results of these experiments are presented in Table~\ref{tab:twoOutputs}.
\begin{table}
\caption{\label{tab:twoOutputs}Relative errors ($e_1$ and $e_2$) for the $m=4$ and $p=2$ case.}
\centering
\begin{tabular}{cc|ccc}
\hline
$r$ & &$k=1$ & $k=3$ & $k=5$  \\
\hline
1 & $e_1$ & 0.361839 & 0.361834 & 0.361626 \\
  & $e_2$ & 0.710212 & 0.710226 & 0.710790 \\ \hline
2 & $e_1$ & 0.043155 & 0.043149 & 0.043112 \\
  & $e_2$ & 0.111431 & 0.111421 & 0.111301 \\ \hline
3 & $e_1$ & 0.004940 & 0.004941 & 0.004861 \\
  & $e_2$ & 0.009017 & 0.009017 & 0.008764 \\ \hline
4 & $e_1$ & 0.003625 & 0.003623 & 0.003628 \\  
  & $e_2$ & 0.020935 & 0.020940 & 0.020898 \\ \hline
5 & $e_1$ & 0.004086 & 0.004087 & 0.004091 \\
  & $e_2$ & 0.018565 & 0.018553 & 0.018534 \\ 
\hline
\end{tabular}
\end{table}
We again see the same trends as the $\epsilon=0.05$ example with improvements in reduced model dimension and transformation degree.  We see the most dramatic improvements with reduced model order with much milder improvements in transformation dimension.  The relative error for each output does not monotonically improve with model order (compare $e_1$ for $r=4$ and $r=5$), but overall, the models are generally better ($e_2$ has more a significant improvement from $r=4$ to $r=5$ in this case).

\section{Conclusions and future directions} \label{sec:conclusions}
We presented a novel and scalable approach for nonlinear balanced truncation of medium to large-scale nonlinear control-affine systems. The approach assumes that system energy functions (e.g., controllability and observability) functions are given in polynomial form. We derive scalable tensor-based formulas for the polynomially-nonlinear state transformation that simultaneously ``diagonalizes'' these relevant energy functions in the new coordinates.  
Since this nonlinear balancing transformation can be ill-conditioned and expensive to evaluate, inspired by the linear case we developed a computationally efficient \textit{balance-and-reduce} strategy. This resulted in further improved scalability and a better conditioned truncated transformation.
We derived closed-form expressions for the resulting ROMs.
We highlight that the work in this paper does not make assumptions on the dynamical systems model form, only on the form of the energy functions (polynomial). Hence the work in this paper applies to general nonlinear systems of control-affine form.

There are several interesting  future research directions to pursue. 
First, the $\mathcal{H}_\infty$ and HJB balancing methods automatically (as a byproduct) give controllers for nonlinear systems, see Part~1 of this paper~\cite[Sec. 2]{KGB_NonlinearBT_Part1} and \cite{sahyoun2013reduced}, \cite{scherpen1996hinfty_balancing}. We plan to evaluate the performance of these controllers on nonlinear systems. 
Second, the evaluation of the nonlinear term and the Jacobian evaluation can be further accelerated by considering empirical interpolation techniques. 
Third, the singular value functions are state-dependent, and cross in state space, necessitating the development of local ROMs. 

\section{Acknowledgements}
We thank Nick Corbin for assisting in the production of Fig.~\ref{fig:balancedManifold} and for valuable comments on drafts of this manuscript. 

\bibliography{NLbal}
\bibliographystyle{abbrv} 

\end{document}